\documentclass[a4paper,french,12pt]{amsbook}
\usepackage[protrusion=true,expansion=true]{microtype}
\usepackage{epsfig}
\usepackage{graphicx}
\usepackage{bourbaki}
\usepackage[all]{xy}
\usepackage[utf8]{inputenc}%           gestion des accents (source)
\usepackage[T1]{fontenc}%              gestion des accents (PDF)
\usepackage[francais]{babel}%          gestion du français
\usepackage{textcomp}% 
\usepackage{eucal}
\usepackage{mathrsfs}
\usepackage{amssymb}
\usepackage{amsxtra}
\usepackage{enumerate}
\makeatletter
\let\@@seccntformat\@seccntformat
\renewcommand*{\@seccntformat}[1]{%
  \expandafter\ifx\csname @seccntformat@#1\endcsname\relax
    \expandafter\@@seccntformat
  \else
    \expandafter
      \csname @seccntformat@#1\expandafter\endcsname
  \fi
    {#1}%
}
\newcommand*{\@seccntformat@section}[1]{%
  \S\csname the#1\endcsname\quad
}
\makeatother

\makeatletter
\renewcommand{\tocsection}[3]{%
  \indentlabel{\@ifnotempty{#2}{\S~\ignorespaces#1 #2.\quad}}#3}
\makeatother

\addto\captionsfrench{

}

\makeatletter
\@addtoreset{section}{part}% ou \numberwithin{section}{part}
\@addtoreset{proposition}{section}% ou \numberwithin{proposition}{part}
\@addtoreset{theoreme}{part}% ou \numberwithin{theoreme}{part}
\makeatother

\input cyracc.def 
 
\newtheorem{theoreme}{Théorème}[chapter]
\newtheorem{proposition}{ Proposition} 
\newtheorem{lemme}{ Lemme}
\newtheorem{corollaire}{Corollaire}

\newtheorem{definition}{Définition}

\theoremstyle{remark}

\newtheorem{exemple}{\it Exemple}

\def \1{\mathbb {1}}
\def \RM{\mathbb {R}}%        corps des reels
\def \NM{\mathbb{N}}%        entiers naturels
\def \ZM{\mathbb{Z}}%        entiers relatifs
\def \CM{\mathbb{C}}%        nombres complexes

\def \QM{\mathbb{Q}}%        nombres rationnels

\def \Aut {{\rm Aut\,}}

\def \p {{\rm exp\,}}
\def \Id {{\rm Id\,}}

\def \d{\partial}%derivee partielle
 
\def\a{\alpha}
\def\b{\beta}
\def\e{\varepsilon}  
\def\g{\gamma}

\def\l{\lambda}

\def\p{\varphi}

\def\G{\Gamma}   
\def\D{\Delta}
\def \s{\sigma}

\def \to{\longrightarrow} 
\def \w{\wedge}
\def \alg{\mathfrak{g}}

\def \< {{\langle }}
\def \> {{\rangle }}
\def \( {\left( }
\def \) {\right) }

\newcommand{\Bt}{{\mathcal B}}
\newcommand{\Ct}{{\mathcal C}}

\newcommand{\Ft}{{\mathcal F}}

\newcommand{\Lt}{{\mathcal L}}
\newcommand{\Mt}{{\mathcal M}}

\newcommand{\Ot}{{\mathcal O}}

\newcommand{\Rt}{{\mathcal R}}

\renewcommand{\mod}{{\rm  mod\,}}
\title[Th\'eorie KAM]{{\sc  Th\'eorie KAM}\\}
\author{ Mauricio  Garay }
 
\begin{document}
 \begin{center}{\LARGE \sc Th\'eorie KAM } \vskip1cm {\Large Mauricio Garay}
 
 \vskip1cm
  \begin{figure}[ht]
  \begin{minipage}[b]{0.48\linewidth}   
     \centerline{\epsfig{figure=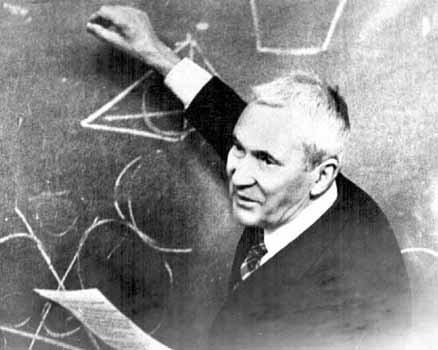,height=0.6\linewidth,width=0.75\linewidth }}
  \end{minipage}
   \begin{minipage}[b]{0.60\linewidth}
     \centerline{\epsfig{figure=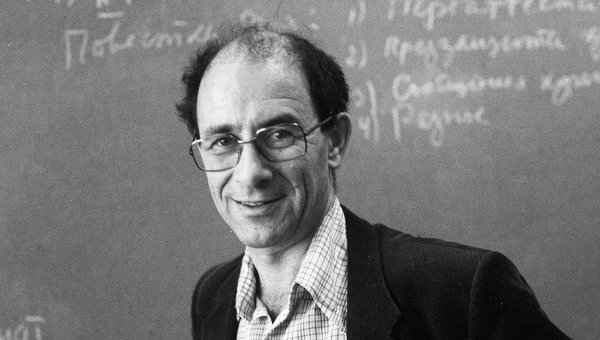,height=0.42\linewidth,width=0.75\linewidth }}
        \end{minipage}\hfill
         \begin{minipage}[b]{0.40\linewidth}
    \centerline{\epsfig{figure= 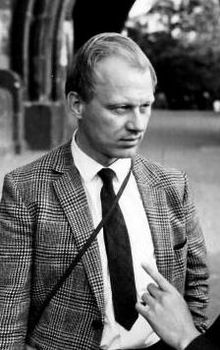,height=0.9\linewidth,width=0.6\linewidth}}
     \end{minipage}
  \end{figure}

  \end{center}
 \tableofcontents
%\begin{flushright}{\em Cet article est dédié à Duco van Straten.}\end{flushright}
%\bigskip

\newpage
\ \\
\thispagestyle{empty}
\vskip3cm
 {\em Ces notes  sont issues d'un cours donné  à l'Université de Ouargla (Algérie) du 08 au 14 Décembre 2013 lors d'une \'Ecole d'hiver en géométrie. Je remercie M. Bahayou et A. Zeglaoui pour cette invitation ainsi que tout ceux qui ont eu la patience de suivre ces exposés, en particulier  Zahia Fernane, Mohamed Kessi, François Laudenbach, Djamel Smai et Nesrine Yousfi. Merci aussi à Duco van Straten et à Jacques Féjoz pour leur aide.}
 
\chapter*{Introduction}
%\begin{flushright}
{\em La première chose à comprendre en sciences et en mathématiques pour faire les premiers pas, c'est de comprendre que l'on comprend très peu!.. Ce n'est vraiment pas simple de comprendre ça.\\}
 Misha Gromov, {\em Entretien radiophonique.}\\
 
%\end{flushright}
Dans le mouvement képlerien,  les planètes décrivent des ellipses parfaites. Tel les rouages d'une horloge, chaque planète possède son orbite, dans un mouvement dont l'harmonie a toujours surpris l'être humain.

Pourtant, ce mécanisme n'est  qu'une approximation, car il ne tient pas compte des influences mutuelles qu'exercent les planètes les unes avec les autres. Lorsqu'au XVIII\up{e}, les mathématiciens énoncèrent la loi Newton à l'aide du calcul différentiel, c'est-à-dire sous la forme que nous les connaissons aujourd'hui, ils se heurtèrent au problème posé par ces perturbations. Celles-ci pourraient-elles entraîner au cours des années des modifications significatives de leur trajectoire ?  

La situation atteint une sorte de paroxysme lorsqu'en 1889, Poincaré démontra que les séries utilisées par les astronomes pour calculer les déviations  aux trajectoires képlerienne étaient divergentes~:  l'influence de petites perturbations donnaient une contribution infinie~!  Quelque chose d'à peine perceptible répété à l'infini finissait par détruire le mouvement harmonieux des astres. Ce phénomène fut interprété comme une manifestation du chaos et une confirmation des hypothèses de la mécanique statistique.  
 
Nous serions tel des êtres microscopiques à l'existence éphémère qui vivant au milieu d'un gaz auraient l'impression d'un mouvement très régulier alors qu'il est en réalité très désordonné. Telle était la vision que pouvait avoir mathématiciens et astronomes du début du vingtième siècle.

Il suffit d'une note de quelques pages pour qu'en 1954, Andreï Kolmogorov bouleverse ce point de vue. Il découvrit que les séries des astronomes n'étaient pas totalement divergentes~: il existait certaines trajectoires particulières pour lesquelles elles pourraient converger. Ces trajectoires se confineraient alors  à un anneau autour de l'ellipse képlerienne. Ainsi, les déviations se compenseraient et les irrégularités constatées par les astronomes pourraient bel et bien se compenser.  Kolmogorov exposa ses résultats au congrès international des mathématiciens de 1954, mais personne ne semblait intéressé par cette découverte. Kolmogorov, lui-même, se tourna rapidement vers d'autres horizons.
  
  Il fallut attendre presqu'une décennie, pour qu'en 1963, un jeune élève de Kolmogorov,  Vladimir Arnold  tente d'appliquer la découverte théorique de son maître au problème du mouvement de la lune. Il découvrit que depuis 1954 aucune démonstration du théorème de Kolmogorov n'avait été établie! Certes, Kolmogorov en avait donné l'esquisse, il avait même donné un algorithme sur lequel se fondait son approche. Mais à aucun moment il n'avait montré que celui-ci convergeait.

En fait, Arnold ne réussit pas à compléter la démonstration de Kolmogorov et, il démontra un résultat beaucoup plus fort, également annoncé par Kolmogorov~: au fur et à mesure que la perturbation diminue, presque toutes les trajectoires restent confinées au voisinage de l'ellipse keplérienne. 

La même année, un autre mathématicien, Jürgen Moser mit au point une technique de démonstration générale pour les problèmes perturbatifs, comme celui rencontré par Kolmogorov. La théorie KAM était née.

Cette naissance a donc précédé celle de toutes les autres théories de déformations et d'espaces de modules alors qu'elle présentait de nombreuses difficultés~: aspect fortement non-linéaire, espaces de modules totalement discontinus etc. L'absence de concepts fondamentaux  a transformé la théorie KAM en une branche technique de l'analyse. 

Pourtant les développements de la topologie et de la géométrie algébrique au siècle dernier nous ont montré que l'on ne saurait se contenter de définir des ensembles (espaces vectoriels, espaces topologiques), il est nécessaire d'introduire des catégories. Ce qui joue un véritable rôle ce n'est pas l'objet, mais le morphisme. L'essence de la théorie KAM et plus généralement de la théorie des perturbations résiderait dans l'étude de certaines catégories d'espaces vectoriels qui vont au-delà de la simple analyse des espaces de Banach ou des chaînes d'espaces de Banach.
%\tableofcontents
%%%%%%%%%%%%%%%%%%%%%%%%%%%%

 %%%%%  
     \chapter{Le th\'eor\`eme de Kolmogorov}
    \section{Champs de vecteurs hamiltoniens}

  Considérons une variété symplectique analytique réelle $(M,\omega)$, c'est-à-dire une variété analytique réelle munie d'une deux-forme $\omega$ analytique telle que le produit intérieur par $\omega$ donne un isomorphisme entre fibré tangent et cotangent~:
   $$TM \to T^*M,\ X \mapsto i_X \omega. $$ 
   Cet isomorphisme permet d'associer à chaque forme différentielle un unique champ de vecteur. Etant donnée une fonction analytique
 $$ H:M \to \RM,$$
 le champ associé à la $1$-forme $dH$ s'appelle le champ {\em hamiltonien} de $H$. Le {\em flot de $H$} est par définition le flot de son champ hamiltonien. 
 
 Le lemme de Darboux dit que toute variété symplectique admet le modèle local $(\RM^{2n},\sum_i dq_i \w dp_i)$. L'isomorphisme entre fibré tangent et cotangent est alors donné par
  $$\d_{q_i} \mapsto dp_i,\ \d_{p_i} \mapsto -dq_i. $$
  Le champ hamiltonien de $H$ est alors
  $$X_H:=\sum_{i=1}^n(\d_{p_i} H \d_{q_i}-\d_{q_i} H \d_{p_i}). $$
Dans ces coordonnées, le champ de vecteur correspond bien aux {\em équations différentielles de Hamilton} enseignées dans les cours de mécanique~:
  $$\left\{ \begin{matrix}\dot q_i&=& \d_{p_i} H \\
  \dot p_i&=& -\d_{q_i} H \end{matrix} \right. $$
 
  %%%%%%%%%%%%%%%%%%%%%%%%
 \section{Int\'egrales premi\`eres}
 \'Etant données deux fonctions
  $$f,g:M \to \RM, $$
  on peut leur associer un crochet de Poisson définit par
  $$\{ f,g \}=\omega(X_f,X_g). $$
   On peut parler de façon indifférenciée du champ hamiltonien de $H$ ou de la dérivation $\{ H,-\}$.
  
  Dans des coordonnées de Darboux, le crochet de Poisson est donné par
  $$\{ f,g \}=\sum_{i=1}^n\d_{p_i} f \d_{q_i}g-\d_{q_i} f \d_{p_i}g.  $$
Plus généralement, une bidérivation est appelée un {\em crochet de Poisson} si elle vérifie l'identité de Jacobi. 

Une $B$-algèbre $A$ munie d'un crochet de Poisson est une {\em algèbre de Poisson}. Un automorphisme est dit de Poisson s'il préserve la structure de Poisson. Par exemple,  la formule précédente définit une structure de Poisson sur $A=\CM[t,q,p]$ qui est $B=\CM[t]$ linéaire. On dira que cette structure est {\em induite} par celle de $\CM[q,p]$.  

  Une quantité $f:M \to \RM$ est préservée par le flot de $H$ si sa dérivée de Lie est nulle, ce qui s'exprime par l'annulation du crochet de Poisson avec $H$.
  $$L_{X_H}f=0 \iff \{ H,f \}=0. $$
  Une telle quantité est appelé une {\em intégrale première}. Par exemple, si $M=\RM^2$ et $H=p$, les seules intégrales premières sont les fonctions de $p$.

Dans tout système hamiltonien, les fonctions de l'hamiltonien  sont trivialement préservées par le flot. En effet, comme le crochet de Poisson est antisymétrique on a~:
  $$ L_{X_H}H=\{ H,H \}=0$$
  et plus généralement $L_{X_H}f(H)=0$.  Dans ses {\em méthodes mathématiques de la mécanique céleste}, Poincaré démontra que, génériquement, ce sont les seules intégrales premières.
  
  Obtenir une intégrale première d'un système hamiltonien revient donc à confiner une solution dans une certaine partie de l'espace. Le théorème de Poincaré semble indiquer donc qu'une particule peut {\em a priori } librement se mouvoir sur sa surface d'énergie sans contrainte. Au début du XX\up{e} siècle, Fermi démontra ce résultat : {\em Pour la plupart des systèmes hamiltoniens, les seules intégrales premières sur un niveau d'énergie sont des fonctions de l'hamiltonien.}
  
Exprimons  la condition de Fermi algébriquement. Une quantité $G$ est préservée sur le niveau d'énergie $H=0$ pourvu que  $\{ H,G \}$ soit une fonction de $H$. On peut le traduire par l'existence d'une fonction $f$ telle que
  $$\{H,G\}=f(H). $$
  Poussant un peu plus loin l'idée de Fermi que l'on est conduit à la théorie KAM.
   
\section{  Id\'eaux invariants }
Les théorèmes de Fermi et de Poincaré sont de nature négative. Ils montrent que sur une hypersurface d'énergie, il n'existe pas d'hypersurface invariante. L'idée de Kolmogorov est de considérer des variétés invariantes de dimension inférieure. 

En géométrie  algébrique, on associe à chaque variété, l'idéal des fonctions qui s'annule sur celle-ci et il est plus pratique d'adopter ce langage algébrique. Celui-ci permet, entre autre, de considérer  des anneaux de séries formelles, de séries analytiques ou de polynômes et d'inclure  le cas  des variétés singulières.

Soient
$$f_1,\dots,f_k:M \to \RM $$
des fonctions analytiques définissant un idéal $I$. \`A cet idéal, on peut associer la variété  (non nécessairement lisse)~:
$$V(I)=\{x \in M:f_1(x)=\dots=f_k(x)=0 \}. $$
On dit que $I$ est {\em radical } si toute fonction s'annulant sur $V(I)$ appartient à $I$. La proposition suivante nous permet d'algébriser la notion de variété invariante~:
\begin{proposition}
Si l'idéal $I$ est radical  alors les assertions suivantes sont équivalentes
\begin{enumerate}[{\rm i)}]
\item $\{ H,I \} \subset I$ ;
\item $V(I)$ est invariante par le flot de $H$.
\end{enumerate}
\end{proposition}
\begin{proof}
On note $f_1,\dots,f_k$ des générateurs de $I$.\\
$ i) \implies ii)$.\\
Notons $\p_t$ le flot de $H$ au temps $t$.
Soit $x$ un point de $V(I)$, on a~:
$$f_i \circ \p_t=e^{t\{ H,-\}}f_i \in I $$
donc si $f_1(x)=\dots=f_n(x)=0 $ alors
$$f_i \circ \p_t(x)=0. $$
$ ii) \implies i)$.\\
Pour tout $x$:
$$ f_1(x)=\dots=f_k(x)=0 \implies f_1(\p_t(x))=\dots=f_k(\p_t(x))=0 $$
et par conséquent, pour tout $i$~:
$$\frac{d}{dt}_{\mid t=0} f_i(\p_t(x))=\{ H,f_i\}(x)=0. $$
Par conséquent les fonctions $\{ H,f_i \}$ s'annulent sur $V(I)$. Comme $I$ est radical cela entraîne que les $\{H,f_i \}$ sont dans $I$.
\end{proof}
Pour simplifier, commençons par considérer le cas polynômial 
$$M=\RM^{2n}$$ muni de coordonnées $q_1,\dots,q_n,p_1,\dots,p_n$ et notons
$$\RM[ q,p ]:=\RM[q_1,\dots,q_n,p_1,\dots,p_n] .$$ 
La condition $\{ H,I \} \subset I$ entraîne que si $I$ est $H$-invariant alors la dérivation
$$\RM[q,p ]/I \to \RM[q,p]/I,\ f \mapsto \{ H,f \}  $$
est bien définie. Comme $\RM[q,p ]/I$ est l'anneau des fonctions sur $V(I)$, cette dérivation est associée à la restriction du champ hamiltonien de $H$ à $V(I)$.

  La remarque fondamentale que nous appliquerons à de nombreuses reprises est qu'{\em en restriction à $V(I)$ le flot hamiltonien reste inchangé si on lui ajoute   un élément de $I^2$.} En effet, posons
  $$H'=H+\sum_{ij}a_{ij}f_if_j $$
 Pour tout polynôme $g$ on a~:
  $$\{ H+\sum_{ij}a_{ij}f_if_j ,g \} =\{ H ,g \}+  \sum_{ij}a_{ij}f_i \{f_j ,g \} +\sum_{ij}a_{ij}f_i\{ H+f_j ,g \}  .$$
  Les polynômes $H$ et $H'$ définissent donc la même dérivation de $\RM[q,p ]/I$.
   
\noindent  {\em Exemple 1.} Considérons les hamiltoniens 
$$H:\RM^2 \to \RM,\ (q,p) \mapsto p$$ 
et $H'=p+qp^2$.   L'idéal $I$ engendré par $p$ est $H$-invariant. Comme 
$$H'=H\ \mod I^2,$$
il est également $H'$-invariant. Ici, la variété
$$V(I)=\{ (q,p) \in \RM^2: p=0 \} $$
est une droite paramétrée par $q$. Les champs hamiltonien de $H$ et $H'$ sont respectivement
\begin{align*}
X_H&=\d_q \\
X_{H'}&=(1+2pq)\d_q-p^2\d_p
\end{align*}
Dans le plan, ce sont des champs différents, mais ils sont tous deux égaux à $\d_q$  en restriction à la droite $V(I)$. 
 
 \noindent  {\em Exemple 2.} Le fibré cotangent au cercle $T^*S^1$ est muni d'une structure symplectique canonique que l'on peut identifier à la forme
 $dq \w dp $ sur $\RM/2\pi \ZM \times \RM$ avec $q \in \RM/2\pi \ZM$, $p \in \RM$. Soit $I$ l'idéal engendré par $p$. Le flot du Hamiltonien  
$$H:\RM/2\pi \ZM \times \RM \to \RM,\ (q,p) \mapsto p$$
décrit des cercles de hauteur constante. Le champ hamiltonien de  
$$H'=p+p^2\sin q $$
est égal à $H$ en restriction à $V(I)$. En particulier,  $V(I)$ est une variété invariante de $H'$.
  \begin{figure}[!htb]
  \centering
  \includegraphics[width=7cm]{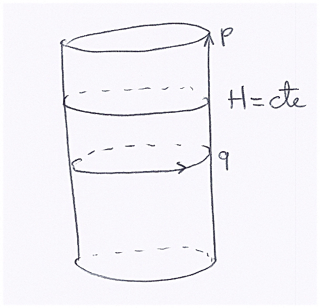}
\end{figure}

  \section{Mouvements quasi-p\'eriodiques}
  Voyons maintenant comment se formule algébriquement la condition pour un système hamiltonien de posséder une tore invariant.
 Le tore de dimension $n$
   $$\mathbb{T}^n=\underbrace{S^1 \times \cdots \times S^1}_{n\ fois}$$
   possède un fibré cotangent trivialisable. En effet, si l'on note $q_i \in \RM/2\pi \ZM$ la \og coordonnée \fg sur le i\up{e} cercle et $dq_i$ la forme différentielle associée. Tout champ de vecteur s'écrit sous la forme 
   $$\sum_{i=1}^n a_i(q) d{q_i}. $$
Le  fibré cotangent est  donc isomorphe à $\mathbb{T}^n\times \RM^n$. Par la suite, nous identifierons la section nulle du fibré avec le tore $\mathbb{T}^n$.
    
Tout fibré cotangent est muni d'une structure symplectique canonique standard et on peut trivialiser le fibré cotangent au tore à l'aide des \og coordonnées action-angle \fg~$\p_1,\dots,\p_n,p_1,\dots,p_n$ avec $\p_i \in \RM/2\pi \ZM$ et $p_i \in \RM$, de telle sorte que la forme symplectique s'écrive
   $$\omega=\sum_{i=1}^n d\p_i \w dp_i .$$
  
Les angles n'étant définis qu'à $2\pi$ près ce ne sont pas des  coordonnées au sens strict, mais on peut poser~:
  $$q_i=e^{\sqrt{-1}\p_i} .$$
 On a alors~:
  $$\sqrt{-1} \frac{dq_i}{q_i}= d\p_i$$
et par conséquent~:
  $$\omega=\sqrt{-1}\sum_{i=1}^n  \frac{dq_i}{q_i} \w dp_i .$$
  Considérons, à titre d'exempl, l'anneau $A=\CM[q,q^{-1},p]$. Cest à dire l'anneau des ponymômes trigonométriques dans les variables $\p_i,p_i$.
  
 La section nulle du fibré cotangent est un tore d'équations
 $$f_i:T^*\mathbb{T}^n \to \RM, (q,p) \mapsto p_i,\ i=1,\dots,n.$$
 Notons $I \subset A$ l'idéal engendré par les $f_i$. On vérifie facilement que c'est un idéal radical.
Il revient au même de dire que la section nulle est un tore invariant que
$$\{ H,I\} \subset I $$
ou encore que $H$ est de la forme
$$H=\a_0+\sum_{i=1}^n \a_i p_i\ \mod I^2. $$ 

C'est la formulation utilisée par Kolmogorov dans son article de 1954, que nous avons donc conceptualisée.

Le vecteur $\a=(\a_1,\dots,\a_n)$ s'appelle {\em vecteur des fréquences}. Sur $V(I)$, les équations de Hamilton se réduisent à
 $$\left\{ \begin{matrix}\dot q_i&=& \a_i\\
  \dot p_i&=& 0\end{matrix} \right. $$ 
  
On peut les intégrer facilement on a 
$$\left\{ \begin{matrix}q_i&=&q_i(0)+ \a_it\ \mod 2\pi\\
  p_i&=& 0\end{matrix} \right. $$
Un tel mouvement s'appelle {\em quasi-périodique}. Si les $\a_i$ sont rationnels alors les trajectoires sont périodiques. En revanche, 
si les $\a_i$ sont linéairement indépendants la trajectoire est dense dans le tore.  

  \begin{figure}[!htb]
  \centering
  \includegraphics[width=7cm]{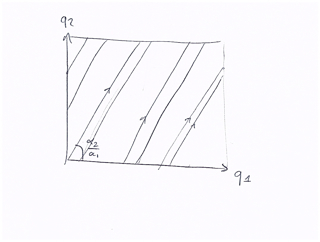}
\end{figure}

\section{Conditions arithm\'etiques}
Le théorème KAM  garantit qu'un mouvement quasi-périodique persiste sous l'effet de perturbations, si sa fréquence est mal approchée par les trajectoires rationelles. Précisons cette notion~: {\em un vecteur $a=(a_1,\dots,a_n)$ est dit diophantien s'il existe
des constantes $(C,\nu)$ telles que~:}
$$\forall j \in \ZM^n,\ |(j,a)|\geq \frac{C}{\| j \|^{n+\nu}} .$$
Un théorème classique de Dirichlet dit que l'on peut toujours obtenir des approximations jusqu'au degré $n$~:
$$\forall a \in \RM^n, \forall j \in \ZM^n,\exists C>0,\  |(j,a)|\leq \frac{C}{\| j \|^{n}} .$$
La condition diophantienne dit qu'on ne peut pas faire mieux que le théorème de Dirichlet.

La question se pose de savoir si de tels vecteurs existent et s'ils sont nombreux. Un nombre trop bien
approché par des rationnels comme
$$l=\sum_{n \geq 0}10^{-n!} $$
définit un vecteur $\a=(1,l)$ qui
ne vérifie pas de condition diophantienne. En effet, les nombres rationnels
$$ l_N=\sum_{n = 0}^N10^{-n!} $$ 
vérifient 
$$| l-l_N| \leq 2\cdot 10^{-(N+1)!}. $$
Posons: 
$$\b_N=(\sum_{n = 0}^N10^{(N-n)!},10^{N!}) \in \ZM^2.$$
On a
$$(\a,\b_N) \leq 2 \cdot (10^{-N!})^N$$
et 
$$\| \b_N \| \geq 10^{N!} $$
donc
$$(\a,\b_N) \leq 2\| \b\|^{-N}. $$
Il existe donc des vecteurs non diophantiens. Cependant,  un résultat dû à Liouville montre que pour tout les nombres algébriques non rationnels $\a \in \overline{\QM} \setminus \QM$, les vecteurs non nul sur la droite $(1,\a)$ sont diophantiens. Mais comme $\overline{\QM}$ est dénombrable, l'ensemble de tous ces vecteurs forment un ensemble de mesure nulle. Cependant~:
\begin{proposition} Pour tout $\nu>0$ fixé, l'ensemble
$$\Omega_\nu =\{ a \in \RM^n: \exists C, \forall j \in \ZM^n,\ |(j,a)|\geq \frac{C}{\| j \|^{n+\nu}} \}  $$
est de mesure pleine.
\end{proposition} 
\begin{proof}

  \begin{figure}[!htb]
  \centering
  \includegraphics[width=7cm]{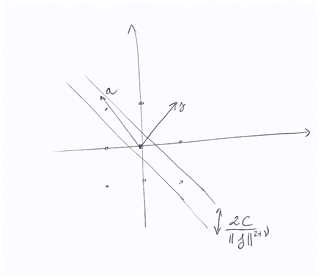}
\end{figure}

Faisons la démonstration pour $n=2$, le cas général est identique. 
 Commençons par fixer la constante $C$ et considérons les ensembles
$$\Omega_{\nu,C}=\{ a \in \RM^2:  \forall j \in \ZM^2,\ |(j,a)|\geq \frac{C}{\| j \|^{2+\nu}} \} . $$
De telle sorte que
$$\Omega_\nu =\bigcup_{C>0}  \Omega_{\nu,C}.$$
Le complémentaire de chacun des ensembles $\Omega_{\nu,C}$ est réunion sur les $j \in \ZM^2$ des ensembles
$$B_C(j) =\{ a \in \RM^2:  |(j,a)|< \frac{C}{\| j \|^{2+\nu}} \}.$$
Or l'ensemble $B_C(j)$ est une bande dont la largeur est 
$$ \frac{2C}{\| j \|^{2+\nu}}.$$
Son intersection  avec le carré $\Ct_N=[-N,N]^2$ a donc une aire majorée par
$$ \frac{4CN}{\| j \|^{2+\nu}}.$$
La suite, à deux indices, de terme général $\| j \|^{-2-\nu},\ j \in \ZM^2$ est sommable pourvu que $\nu>0$ , nous notons $K_\nu$ sa somme.
Le complémentaire de $\Omega_{C,\nu} \cap \Ct_N$ est la réunion des $B_C(j) \cap \Ct_N$~:
$$\Ct_N \setminus \left( \Omega_{C,\nu} \cap \Ct_N \right)=\bigcup_{j \in \ZM^n} B_C(j) \cap \Ct_N $$
où $Vol$ désigne la mesure de Lebesgue.
Ce dernier possède une aire  majorée par  la somme des aires des $ B_C(j) \cap \Ct_N$ et donc par~:
$$Vol(\bigcup_{j \in \ZM^n} B_C(j) \cap \Ct_N) \leq \sum_{j \in \ZM^2 \setminus \{ 0\}} \left( \frac{4CN}{\| j \|^{2+\nu}} \right) =4K_\nu CN.$$
Nous obtenons ainsi l'inégalité~:
$$Vol(\Ct_N \setminus \Omega_\nu)=Vol(\bigcap_{C >0} \Ct_N \setminus \left( \Omega_{C,\nu} \cap \Ct_N \right)) \leq \lim_{C \mapsto 0}4K_\nu CN=0  .$$
Ceci montre que la réunion des $\Omega_{C,\nu}$ est de mesure pleine et achève la démonstration.
\end{proof}
 Cette proposition montre que bien qu'il soit difficile de donner explicitement un nombre diophantien non algébrique, ceux-ci sont en fait très nombreux.
 
  \section{Le th\'eor\`eme des tores invariants}
Nous pouvons maintenant énoncer le théorème de Kolmogorov. Pour cela, nous allons utiliser la notion de germe.
Soit $M$ un espace topologique et $K \subset M$ un sous-ensemble. Deux
 sous-ensembles $V,V'$ de $M$ ont le même germe en $K$ s'il existe un voisinage $U$ de $K$ tel que 
 $$V \cap U=V' \cap U.$$
 On définit ainsi le germe de l'ensemble $V$ en $K$ noté $(V,K)$.

 Soit maintenant $U,V$ des voisinages de $K$ dans $M$ et $X$ un ensemble.  Deux applications
 $$f:U \to X,\ g:V \to X$$
 définissent le même germe en $K$ si elles sont égales sur $U \cap V$. La classe d'équivalence définie par une fonction s'appelle le germe de l'application en $K$. Dans le cas analytique, le passage au germe consiste simplement à \og oublier \fg l'ensemble sur lequel la fonction est définie. Ainsi les applications
 $$\CM \to \CM, z \mapsto z^2 $$
 et 
 $$D(0,1) \to \CM,\ z \mapsto z^2,\ D(0,1):=\{ z \in \CM: | z | <1 \} $$
 définissent le même germe en tout point du disque et même sur tout sous-ensemble compact du disque. La  notion de germe permet de raccourcir les énoncés car on s'affranchit de préciser les voisinages sur lesquels sont définis les
objets.
  
  L'anneau des germes de fonctions analytiques réelles sur $M$ le long de $K$ est noté $\Rt_{M,K}$.  
 On note  $\CM\{t\}$ l'espace vectoriel des germes en l'origine dans $\CM$, c'est-à-dire l'espace vectoriel des séries convergentes en une variable $t$.

   \begin{theoreme} Munissons $M= \RM \times T^*\mathbb{T}^n=\{ t,q,p \}$ de la structure de Poisson induite par celle de $T^*\mathbb{T}^n$ et posons $K=\{ 0 \} \times \mathbb{T}^n $. Soit $I$ l'idéal engendré par les $p_i$
et   $H(t,q,p) \in \Rt_{M,K}$  tel que
$$H(t,q,p)=\sum_{i=1}^n \a_i p_i \ \mod I^2 .$$
Supposons que
   \begin{enumerate}[{\rm i)}]
   \item le vecteur $\a=(\a_1,\dots,\a_n)$ soit diophantien ;
   \item la matrice hessienne $\d_{p_ip_j} H(0)$ soit non-degénérée
   \end{enumerate}
   alors il existe un automorphisme de Poisson $\p \in \Aut(\Rt_{M,K})$ 
   tel que
$$\p(H)  =H\ \mod (I^2 \oplus \CM\{ t \}) .$$
    En particulier, $H$ possède un idéal invariant isomorphe à $I$ et, par conséquent tout représentant de $H$  admet une famille de tores invariants paramétrée par $t$, pour $t$ assez petit.
   \end{theoreme}
   
  Notre but va être de donner une démonstration conceptuelle de ce théorème qui servira de base à une théorie générale des déformations et des formes normales.
      
  Considérons le cas $n=1$. Posons
  $$H(t,q,p)=\a p +a(q,p,t) p^2+tb(q,p) . $$
  Le cercle d'équation $p=t=0$ est invariant et le flot hamiltonien est de période $2\a\pi$. Comme l'énergie est conservée les cercles
  $$V_{t,\e}:=\{ (q,p) \in T^* S^1:H(t,q,p)=\e \} $$
  sont des tores invariants de dimension 1. Le théorème dit que pour tout $t$, on peut trouver $\e$ pour que le flot de $H_t$ sur 
  $V_{t,\e}$ soit conjugué à celui de $H_0$ sur $V_0$. La condition de non dégénérescence est nécessaire pour obtenir la conclusion du théorème. En effet,
la famille 
  $$H(t,q,p)=\a p +tp  $$
donne un exemple pour les $H(t,-)$ ont tous des flots non conjugués, car de fréquence différente. Cependant, cette famille possède
des tores invariants pour tout $t$ bien que le théorème ne s'applique pas. L'affaiblissement de la condition de non-dégénérescence possède une longue histoire (voir bibliographie).
   
A partir de $n=2$, le théorème devient non trivial.  Il affirme, par exemple, que tout hamiltonien de la forme
$$H(t,q,p)=p_1+\a p_2+p_1^2+p_2^2+tR(t,q,p) $$  
possède des tores invariants pour tout $t$ suffisamment petit, dès que $\a \in \overline{\QM} \setminus \QM$.  
    
 \chapter{Le probl\`eme fondamental de la m\'ecanique} 
 \section{Formulation abstraite du th\'eor\`eme de Kolmogorov} 
 Considérons à nouveau, l'algèbre $\Rt_{M,K}$ des fonctions analytiques au voisinage du tore dans les variables $q_i,p_i,t$. Notons $I$ l'idéal engendré par les $p_i$, fixons un vecteur $\a \in \RM^n$ diophantien et posons, comme précédemment
 $$H_0=\sum_{i=1}^n \a_i p_i+\sum_{i=1}^n \b_{ij} p_ip_j$$
 avec $\b=(\b_{ij}) \in M(n,\RM)$.  La section nulle est un tore invariant de fréquence $\a=(\a_1,\dots,\a_n)$ pour $H(t=0,-)$, ce qui se traduit algébriquement par
 $$\{ H,I \} \subset I. $$ 
 
 Considérons le sous-espace affine $H+M \subset E $ avec
 $$M=t \Rt_{M,K}. $$
Les éléments de $H \in H_0+M$ sont les déformations de $H_0$.

Posons maintenant $F=\RM\{ t \}+I^2 $. Pour tout élément du sous-espace affine $H_0+tF$ admet $I$ pour idéal invariant. 

 Le groupe des automorphisme de Poisson agit sur l'algèbre de Poisson $\Rt_{M,K}$, on note $G$ le stabilisateur de $H+M$.
 Le théorème de Kolmogorov alors s'énonce sous la forme suivante~:{ \em  Si  la matrice $(\b_{ij})$ est non-degénérée
 et si le vecteur $\a$ est diophantien alors
  tout élément  de $ H +M $ est dans l'orbite d'un élement de $H+F$, autrement dit l'application
  $$G \times F \to H+M,\ (\p,\a) \mapsto \p(H+\a) $$
  est surjective. }
   
   Ce type de problématique peut-être considérer de manière générale lorsqu'un groupe agit sur un espace topologique.
   On dit alors $F$ est une transversale pour l'action au point considéré. Lorsque l'on peut prendre $F=\{ 0 \}$ on dit que l'espace est localement  $G$-homogène.   On peut schématiser la situation par la figure suivante~:
   \begin{figure}[!htb]
  \centering
  \includegraphics[width=5cm,height=4cm]{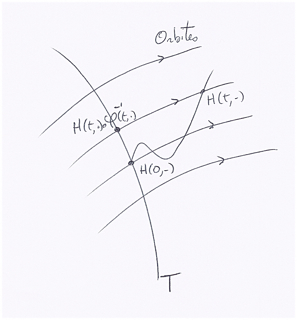}
\end{figure}

 \section{Actions de groupes de Lie}
 Lorsqu'un groupe de Lie $G$ agit sur une variété $V$, en tout point $x \in V$, il définit une action infinitésimale de l'algèbre de Lie $\alg$ sur l'espace tangent en $x$.
 En effet, en différentiant l'action
 $$\rho:G \to M,\ g \mapsto g \cdot x. $$
  en l'identité, on trouve une application
  $$\rho_*:\alg \to T_x M,\ \xi \mapsto D\rho(\Id)\xi $$
  Lorsque $M$ est un ouvert d'un espace vectoriel, on identifie $T_xM$ avec l'espace vectoriel.  L'orbite du point $x$ sous l'action du groupe $G$ est alors tangente à celle de l'algèbre de Lie.
  
  Commençons par l'exemple de  $G=GL(n,\CM)$ agissant sur lui-même  par conjugaison
  $$P \cdot A:=PAP^{-1} .$$ 
  Son algèbre de Lie est l'espace des matrices muni du crochet
  $$[A,B]=AB-BA. $$
   Posons 
  $$P=e^{t B}=\Id+tB+o(t),$$ on a alors
  $$PAP^{-1}=(\Id+tB)A(\Id-tB)+o(t)=A+t[B,A]+o(t).$$
  L'action de l'algèbre de Lie est alors
  $$B \mapsto [B,A]. $$
 Supposons que la matrice $A$ soit diagonale avec des valeurs propres distinctes. Toute matrice $B$ dans un voisinage suffisament petit de $A$ a ses valeurs propres distinctes, elle est donc diagonalisable~:
  $$\exists P \in GL(E); D=PBP^{-1} .$$
Autrement dit l'espace vectoriel des matrices diagonales $F$ est une transversale de l'action au voisinage de la matrice $A$.
  
  \begin{figure}[!htb]
  \centering
  \includegraphics[width=7cm]{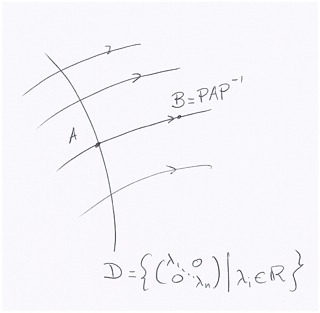}
\end{figure}

  Reprenons l'exemple de l'action adjointe.  
  L'action adjointe possède une linéarisation intermédiaire. En effet, comme l'action fixe l'identité
  $$P \cdot \Id=\Id ,$$
  la différentielle de l'action
  $$G \to G,\ M \to P \cdot M $$
  en l'identité envoie l'algèbre $\alg$  sur elle-même. On obtient ainsi un représentation linéaire du groupe $G$ appelée {\em représentation adjointe}~:
 $$G \to gl(\alg) .$$
 Lorsque $G$ est un groupe de matrices, la conjugaison
 $$M \mapsto PMP^{-1} $$
 est linéaire donc égale à sa dérivée.
 
 Prenons à présent $G=SO(3,\RM)$ le groupe des rotation du plan
 $$G=\{ A \in M(3,\RM): ^t A A=\Id. \} $$
 En subsituant $A=\Id+tX+o(t)$ avec $X \in \alg$, dans cette égalité on trouve~:
 $$\Id+t(^t X+X)+o(t)=\Id. $$
 Ce qui nous dit que l'algèbre de Lie de $G$ est formé des matrices anti-symétriques. C'est un espace vectoriel de dimension $3$ et l'isomorphisme
 $$\p:\RM^3 \to \alg (x,y,z) \mapsto  \begin{pmatrix} 0 & x & y \\ -x & 0 & z \\ -y &-z & 0 \end{pmatrix}$$
 envoie le produit vectoriel sur le crochet. On vérifie facilement que, via cet isomorphisme, la conjugaison de $\p(\Omega)$ par une rotation $P$ correspond à prendre l'image de $\Omega$ par $P$~:
 $$P\p(\Omega)P^{-1}=\p(P\Omega). $$
 La représentation adjointe de $SO(3)$ s'identifie donc avec sa représentation naturelle dans $\RM^3$. Les orbites, en dehors de celle de l'origine, sont des sphères. Toute droite issue de l'origine définie une transversale à chacune de ces orbites au point d'intersection.

  \section{L'algorithme de Kolmogorov}
 Nous allons à présent résoudre le problème de Kolmogorov abstrait en dimension finie. Notre travail sera ensuite de  développer des techniques pour que la démonstration puisse se transposer au cas de la dimension infinie. 
 
 Il s'agit de montrer que localement, une transversale à l'action de l'algèbre de Lie donne une transversale à l'action du groupe. Nous commençons par le cas homogène.
\begin{theoreme} Soit $G$ un groupe de Lie agissant sur une variété $M$ et $a \in M$. Si l'action infinitésimale 
$$ \alg  \to  T_p M,\ \xi \mapsto \xi \cdot \a $$
est surjective alors $M$ est localement $G$-homogène au voisinage de $x$.
\end{theoreme}
C'est une conséquence directe du théorème des fonctions inverses. En effet, sous les hypothèses du théorème la différentielle en l'identité
de
$$G \to M,\ g \mapsto g \cdot a $$
est surjective, donc cette application est surjective sur un voisinage de l'identité.

L'algorithme de Kolmogorov est une variante de la méthode de Newton qui permet de construire pour chaque $x $ dans un voisinage assez petit de $a$ une suite $g_n$ (dépendant de $x$) qui converge rapidement vers $g$ tel que
$$g\cdot a=x. $$

Pour simplifier, nous nous restreignons aux cas de l'action naturelle d'un groupe de matrice $G \subset GL(V)$ sur un espace vectoriel $V$.

On note
$$j:V  \to \alg$$
un inverse à droite de 
$$\rho:\alg \to V,\ \xi \mapsto \xi(v)$$

On fixe $x_0 \in V$. Par hypothèse, on peut écrire $x_0 \in V$ sous la forme
$$x_0=\xi_0 (a). $$
On pose 
\begin{align*}
x_1=&\, e^{-\xi_1}x_0=x_0-\xi_1(a)+o(\| \xi_1\|)=o(\| \xi_1\|) \\
\xi_1=&\,j(x_1)
\end{align*}
 
 On définit ainsi, de proche en proche, des suites $(x_n),(\xi_n)$ telles que~:
$$\left\{ \begin{matrix} \xi_n&=&j(x_n) \\ 
x_{n+1}&=&e^{-\xi_n}( x_n) \end{matrix} \right.$$
On a ainsi
$$x_{n+1}=e^{-\xi_n} (x_n) =e^{-\xi_n}e^{-\xi_{n-1}}x_{n-1} =\dots=\prod_{i \geq 0}^{n}e^{-\xi_i}x_0$$
Le premier problème qui se pose est de donner un critère pour la convergence d'un produit infini d'exponentielle~:
 munissons l'espace des endomorphisme de $\RM^n$ de la norme d'opérateur~:
 $$\| \xi \|:=\sup_{x \in \RM^n} \frac{\| \xi(x) \|}{\| x \|} .$$
 \begin{lemme} Soit $\xi_i:\RM^n \to \RM^n$ des applications linéaires. Le produit infini $\prod_{i \geq 0}e^{\xi_i}$
 est convergent si et seulement si la suite $(\| \xi_i \|)$ est sommable.
 \end{lemme}
 \begin{proof}
 En effet
 $$\log \| \prod_{i \geq 0}e^{\xi_i} \| \leq \log  \prod_{i \geq 0}e^{\| \xi_i \|}=\sum_{i \geq 0}\| \xi_i \|. $$
 \end{proof}
La généralisation de ce lemme à la dimension infinie jouera un rôle fondamental. Lorsque le lemme s'applique la limite de la suite
$$g_n:= \prod_{i \geq 0}^{n}e^{-\xi_i}$$
converge vers un élément $g$ tel que $g \cdot a=x$. 

Tâchons maintenant de ré-écrire la suite $(\xi_n)$ comme itération d'une fonction. Comme
 $$x_n=\xi_n(a).$$
l'expression 
$$x_{n+1}=e^{-\xi_n} (a+x_n),$$ 
on a~:
$$x_{n+1}=e^{-\xi_n} (a+\xi_n( a)).$$
Ce que l'on peut ré-écrire sous la forme
$$x_{n+1}=(e^{-\xi_n}(\Id+\xi_n)-\Id)( a) .$$

La suite $(\xi_n)$ s'obtient en itérant la fonction $F=j \circ f$ où $f$ est la fonction analytique
$$ x \mapsto e^{-x}(1+x)-1$$
 avec un point critique à l'origine. Le théorème du point fixe suivant montre que la convergence de notre algorithme est très rapide~: \begin{theoreme}  Soit
$$F:\RM^n \to \times \RM^n,\ x \mapsto F(x),\ F(0)=0  $$
une fonction $C^2$ avec un point critique à l'origine. Il existe alors un voisinage $U \subset \RM^n$ de l'origine et une constante $\rho<1$
tell que pour tout $x \in U$, la suite $x_n=F^n(x)$ converge et
$$\| x_n \| \leq \rho^{2^n}. $$
\end{theoreme}
\begin{proof}
Posons
$$B(\e)=\{ x \in \RM^n:\| x \| \leq \e \}. $$
La formule de Taylor implique l'existence  d'une boule  $B(R)$  et d'une constante $C \in \RM$ telles que~:
$$\| F(x) \|  \leq C \| x \|^2 $$
pour tout $x \in B(R)$. Soit $r$ suffisamment petit pour que
$$\rho:=C r<1$$

Une récurrence sur $n$ montre que, pour tout $x \in B(r)$, on a~:
$$\| x_n \| \leq \rho^{2^n} .  $$
En effet, la formule de Taylor nous donne
$$\| x_{n+1} \| = \| F(x_n)  \| \leq \rho   C\| x_n \|^2 \leq Cr \| x_n\| \leq  \rho^{2^{n+1}}.$$
Ce qui achève la démonstration théorème.
\end{proof}

 \section{Cas g\'en\'eral}
 Le cas non-homogène n'est qu'une variante à paramètre du cas homogène~:
 \begin{theoreme} Soit $n>0$ et $G \subset GL(n,\RM)$ un groupe agissant sur une variété $M$. Soit $F \subset M$ une sous-variété telle que l'action infinitésimale induise une application surjective
$$T \times \alg  \to T_a M/T_a F,\ (\a,\xi) \mapsto (\a,\overline{\xi \cdot \a}) $$
alors l'application
$$T \times G   \mapsto V,\ (\a,g) \mapsto g \cdot \a $$
est surjective au-dessus de tout voisinage suffisamment petit de a.
\end{theoreme}
C'est à nouveau une application du théorème des fonctions inverse. Comme précédemment nous souhaiterions une construction explicite
de $g$ et $\a$ à l'aide d'un algorithme rapide.

Pour cela, nous avons besoin d'un théorème de point fixe à paramètre~:
 \begin{theoreme} Soit
$$F: \RM^k \times \RM^n \to \RM^k \times \RM^n,\ (a,x) \mapsto (a+f_1(a,x),f_2(a,x)),\ f_i(-,0)=0  $$
une application $C^2$ telle que $f_1(a,-)$ et $f_2(a,)$ possèdent des points critiques en l'origine pour tout $a$.
Pour tout $a$ et pour tout $\rho<1$, il existe un voisinage $B$ de l'origine dans $\RM^n$ tel que pour tout $x \in B$
La suite $(a_n,x_n)=F^n(a,x)$ est convergente et~:
$$\| x_n \| \leq \rho^{2^n}. $$
\end{theoreme}
\begin{proof}
On choisit des boules $B_n(r) \subset \RM^n$, $B_k(R) $, $r \leq R<1$, telles que
$$\sup_{b \in a+B_k(R)}\left( \| f(b,x)\|+\| g(b,x)\| \right) \leq C\| x \|^2 $$
pour tout $x \in B_n(r)$ avec $\rho:=Cr<R$. 
Montrons par récurrence que pour $n \geq 1$, on a~:
$$\| x_n \| \leq \rho^{2^n},\ \| a_n \| \leq \| a_0\|+  R \sum_{i=0}^{n-1} \rho^{2^i}   $$
dès que $\sum_{i \geq 0} \rho^{2^i} \leq 1$

Comme dans le théorème du point fixe précédent, la formule de Taylor nous donne~:
$$\| x_{n+1} \| \leq Cr\rho^{2^n}=\rho^{2^{n+1}}.$$
De même~:
$$ \| f(a_n,x_n)\|  \leq C \| x_n \|^2 \leq Cr \| x_n \| \leq R\rho^{2^n}, $$
et, par conséquent
$$\ \| a_{n+1} \| \leq \| a_n\|+ R\rho^{2^n} \leq \| a_0\|+  R \sum_{i=0}^n \rho^{2^i}. $$
Ceci achève la démonstration théorème.
\end{proof}

Comme précédemment, nous ne considérons que le cas où $M=V$ est un espace vectoriel, $F$ un sous-espace vectoriel et $G$ un sous-groupe de $GL(V)$, qui agit de façon naturelle sur $V$.

On note
$$j:F \times V \to F \times \alg$$
un inverse à droite de 
$$\rho:F \times \alg \to F \times V,\ (\a,\xi) \mapsto (\a,\xi(a+\a)).$$ 

On fixe $x_0 \in V$ et on cherche $g$ tel que
$$g(a)=a+x_0. $$
On commence par écrire $x_0$  sous la forme
$$x_0=\xi_0 ( a_0+x_0)+\a_0,\ \a_0 \in F$$
avec $j(0,x_0)=(\a_0,\xi_0)$ et $a_0=a$.

On pose ensuite
\begin{align*}
a_1=&a_0+\a_0 \\
x_1=&e^{-\xi_0}(a_0+x_0)-a_1. 
\end{align*}

\begin{figure}[!htb]
  \centering
  \includegraphics[width=10cm]{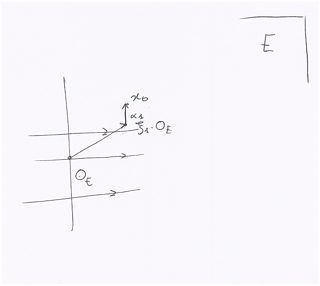}
\end{figure}

On définit ainsi de proche en proche des suites en posant 
\begin{align*}
a_{n+1}=&a_n+\a_n ;\\
x_{n+1}=&e^{-\xi_n}(a_n+x_n)-a_{n+1};\\
(\a_{n+1},\xi_{n+1})=&j(a_{n+1},x_{n+1}). 
\end{align*}
On a alors
$$a_{n+1}+x_{n+1}=e^{-\xi_n}(a_n+x_n)=e^{-\xi_{n}}e^{-\xi_{n-1}}(a_{n-1}+x_{n-1})=\dots=\prod_{i \geq 0}^{n}e^{-\xi_i}x_0$$
 
 Montrons que cette itération est rapide. En remplaçant l'égalité
$$x_n=\xi_{n}( a_n)+ \a_n.$$
dans l'expression 
$$x_{n+1}=e^{-\xi_n} (a_n+x_n )-a_{n+1},$$
on trouve~:
$$x_{n+1}=e^{-\xi_n} (a_n+\xi_n(a_n)+\a_n )-a_{n}-\a_n=\left( e^{-\xi_n} (a_n+\xi_n(a_n))-a_n \right)+\left( e^{-\xi_n}\a_n )-\a_n\right),$$
donc
$$x_{n+1}=(e^{-\xi_n}(\Id+\xi_{n})-\Id) (a_n) +(e^{-\xi_n}-\Id) (\a_n) .$$
L'itération est de la forme
$$(\xi_{n+1},\a_{n+1})=f(a_n,\xi_n,\a_n),\ a_{n+1}=a_n+g(a_n,\xi_n,\a_n)$$
où $f(b,-)$ et $g(b,-)$ possède un point critique en l'origine pour tout $b$.
Un tel procédé itératif converge quadratiquement d'après le théorème du point fixe à paramètre.

 Nous avons démontré la convergence rapide de l'algorithme de Kolmogorov en dimension finie. Cet algorithme repose sur trois ingrédients
 \begin{enumerate}[{\rm i)}]
 \item L'existence d'une exponentielle qui envoie l'algèbre de Lie sur le groupe ;
 \item Un théorème de point fixe.
 \end{enumerate} 
 
 \chapter{Espaces de Kolmogorov}
  %%%%%%%%%%%%%%%%%%%%%%%%%%%%%%
  %%%%%%%%%%%%
 %%%%%%%%%%%%%%%%%%%%%%%%%%%%%
%%%%%%%%%%%%%%%%%%%%%%%%%%%%%
\section{Systèmes directs d'espaces vectoriels}
Soit $I$ un ensemble et $k$ un corps.

La catégorie des espaces vectoriels gradués sur un corps $k$ est naturellement définie~: ses objets sont les $k$-espaces vectoriels gradués
et ses morphismes sont les applications linéaires
$$u: \bigoplus_{i \in I} E_i \to \bigoplus_{j \in J}F_j.$$
Un morphisme est dit gradué si pour tout $i$, il existe un unique $j$ tel que 
$$u(E_i) \subset F_j.$$

Habituellement, on considère des espaces vectoriels graduées par $\NM$ ou $\ZM$ comme l'espace des polynômes $\RM[X]$ ou celui des polynômes de Laurent $\RM[X,X^{-1}]$. Nous allons plutôt considérer des graduations par des segments semi-ouverts $]0,S]$. On peut par exemple prendre 
$$E=\bigoplus_{s \in ]0,S]} E_s $$ 
où les $E_s$ sont les espaces vectoriels de fonctions continues 
$$E_s=C^0([-s,s],\RM). $$ 
Un élement $f \in E$ est une somme finie formelle
$$f=f_{t_1}\oplus \cdots \oplus f_{t_n},\ t_i \in ]0,S]. $$

Ces espaces possèdent non seulement la propriété d'être gradué, mais aussi d'être des naturellement des systèmes directs. Rappelons cette notion. Lorsque $I$ est un ensemble ordonné, un espace vectoriel gradué sur $I$ 
$$ E=\bigoplus_{i \in I} E_i $$
muni d'applications linéaires
$$f_{ij} :E_i \to E_j,\ \forall i>j$$
telles que
$$f_{ij}f_{jk}=f_{ik},\ i <j <k $$
est appelé un {\em syst\`eme direct d'espaces vectoriels}. Lorsqu'on écrit simplement
$$E_i \to E_j $$
sans préciser le morphisme, il est sous-entendu qu'il s'agit du morphisme $f_{ij}$.
Tout système direct possède une limite directe obtenu en identifiant des vecteurs qui ont la même image dans un certain $E_j$ assez grand.

Les systèmes directs forment une catégorie dont les morphismes gradués sont les applications linéaires
$$u:E \to F $$
qui commute aux morphismes du système~:
$$\xymatrix{E_i \ar[r] \ar[d]^{u_i} & E_j \ar[d]^{u_j} \\
 F_k \ar[r]& F_l }$$
 Les morphismes de la catégorie des  syst\`eme direct d'espaces vectoriels sont les sommes finies de morphismes gradués.

 Chaque système direct défini naturellement des sous-systèmes directs qui ont la même limite~:
 pour tout $t \in I$, les espaces vectoriels
 $$E[t]:=\bigoplus_{i \leq  t} E_i  $$
 définissent également un système direct.

On a un foncteur naturel de la catégorie des systèmes directs dans celles des espaces vectoriels
$$SD \to EV,\ (\bigoplus E_i,f_{ij}) \mapsto \underrightarrow{\lim} E_i $$
Par exemple, la limite directe du système
$$E:=\bigoplus_{s \in ]0,S]} C^0([-s,s],\RM) $$ muni des applications de restriction est l'espace des germes de fonctions continues en $0$.

Si l'on considère l'exemple des polynômes de dégré $i$~:
$$E_i=\RM_i[X],\ i \in \NM.$$  Les inclusions
 $$f_{ij}:\RM_i[X] \to \RM_j[X]$$
 font de la somme directe $E=\bigoplus E_i$ un système direct.  La limite directe de $E$ est l'espace des polynômes.
Un élément de $E$ est une somme formelle de polynômes
$$P_1 \oplus \dots \oplus P_n,\ P_i \in \RM_i[X] .$$
Passer à la limite directe revient à remplacer la somme formelle par la somme habituelle des polynômes. La différence entre l'élément du système direct et sa classe d'équivalence dans sa limite directe est donc assez subtile.
 
On définit naturellement la catégorie des systèmes directs d'espaces vectoriels topologiques en se restreignant aux systèmes dont les morphismes sont continus.
 
Si les $E_i$ sont des espaces vectoriels topologiques, on munit $E$ de la topologie produit~: une base de de voisinage de l'origine est donnée par les produits d'ouverts $\prod_i U_i$ où $U_i$ est un ouvert de $E_i$ avec $U_i=E_i$
sauf pour un nombre fini d'indice. Si $E_i$ est normé, les bornés de cette topologie sont de la forme $\prod_i B_i$ où $B_i$ est borné dans $E_i$.
 Par ailleurs, lorsque $E$ et $F$ sont   des espaces vectoriels topologiques, l'espace $\Lt(E,F)$ est muni d'une topologie que l'on peut décrire ainsi.
On se donne un ouvert $U$ de $E$ et un borné $B$ de $F$. Les ensembles de la forme
$$\Omega(U,F)=\{ v \in L(E,F): v(U) \subset B \} $$
 définissent une base d'ouverts. 
%%%%%%%%%%%%%%%%%%%%%%%%%%%%%%%%`
\section{La cat\'egorie des espaces de Kolmogorov}
\label{SS::definition}
Considérons à présent le cas particulier où $I$ est un intervalle $]0,S]$ avec $S>0$. La catégorie des $S$-espaces de Kolmogorov est une sous catégorie pleine des systèmes directs d'espaces vectoriels topologiques.

\begin{definition} Un $S$-espace de Kolmogorov $E$ est un système direct d'espaces de Banach $(E_s,| \cdot |_s)$ paramétré par $]0,S]$ tel que les
les applications linéaires
$$f_{ts}:E_t \to E_s$$ soient continues de norme au plus $1$ pour tout $s<t$.
\end{definition}
Très souvent, on ne précisera pas la valeur $S$.  Les morphismes de la catégorie des espaces de Kolmogorov sont les morphismes continus de systèmes directs. Les espaces vectoriels 
$$E:=\bigoplus_{s \in ]0,S]} C^0([-s,s],\RM), \ S>0 $$ muni des applications de restriction donnent des exemples simples d'espaces de Kolmogorov.

Nous désignerons par $\Lt(E,F)$ l'espace vectoriel des morphismes de $E$ dans $F$ et lorsque $E=F$, nous utiliserons la notation $\Lt(E)$ au lieu de $\Lt(E,E)$. 

On étend la norme de $E_s$ a une application
$$|\cdot |_s:E \to \RM_+ \cup \{+\infty \} $$
de la façon suivante~: pour
$$x=x_1\oplus \cdots \oplus x_n \in {\bigoplus}_{i=1}^n E_{t_i},$$ 
on pose
  $$ | x|_s:=\left\{ \begin{matrix} |f_{t_1s}(x_1)+\dots+f_{t_ns}(x_n)|& \ {\rm si\ t_i \geq s },\ \forall i\\  +\infty &\ {\rm sinon.} \end{matrix} \right. $$
 
 Un espace de Kolmogorov possède de nombreuses filtrations. La plus importante d'entres-elles
  est celle qui généralise la filtration par l'idéal maximal de l'espace des germes~: pour tout espace de Kolmogorov $E$, on définit
   $$E^{(k)}=\{ x \in E: \exists C,\tau,\ | x|_s \leq Cs^k,\ \forall s \leq \tau \}.$$
 On a alors
 $$E:=E^{(0)} \supset E^{(1)}\supset E^{(2)} \supset \cdots. $$
 %%%%%%%%%%%%%%
\section{Exemples clefs}
Nous allons maintenant donner un premier exemple important d'espace de Kolmogorov.
 On fixe $S>0$. Soit $D=(D_s)$ la famille des disques $D_s \subset \CM$  de rayon de $s <S$~:
$$ D_s=\{ z \in \CM : | z | \leq s \}.$$ 
 Les espaces vectoriels des fonctions continues sur $D_s$ et holomorphes dans l'intérieur du disque
$$E_s:=C^0(D_s,\CM) \cap \Ot_{\CM}(\mathring{D}_s) $$ 
sont des espaces de Banach pour la norme
$$| g |_s:=\sup_{z \in D_s} |g(z)|. $$
L'application de restriction
$$f_{ts}:E_t \to E_s$$
est de norme $1$. Nous avons ainsi définit un espace de Kolmogorov associé à la famille $D$ que nous noterons $C^{\omega}(D)$, sans en indiquer la dépendance en $S$.

Comme toute fonction holomorphe est automatiquement dérivable, les opérateurs différentiels permettent de définir des morphismes $C^{\omega}(D) $.  Le foncteur limite directe envoie l'espace de Kolmogorov $C^{\omega}(D).$
sur l'espace vectoriel $\CM\{ z \}$ des séries convergentes en une variable.

Par ailleurs, l'anneau $\CM\{ z \}$ est local, c'est-à-dire qu'il ne possède qu'un idéal maximal
 $$\Mt:=\{ f \in \CM\{ z \}: f(0)=0 \}. $$ 
Le foncteur limite directe envoie $E^{(k)}$ sur la puissance $k$-ième de $\Mt$. En effet, si
$$f =\sum_{j \geq k} a_j x^j \in \Mt^k$$
on a
$$| f(z)|_s =| a_j| s^k+o(s^k). $$

 Considérons maintenant les espaces vectoriels des fonctions $L^p$-intégrable sur $D_s$ et holomorphes dans l'intérieur du disque
$$E_s':=L^2(D_s,\CM) \cap \Ot_{\CM}(\mathring{D}_s). $$ 
L'application de restriction
$$f_{ts}:E_t' \to E_s'$$
est de norme $1$. Nous avons ainsi définit un autre espace de Kolmogorov associé à la famille $D$ que nous noterons $L^{p,\omega}(D).$
 Le foncteur limite directe envoie également l'espace de Kolmogorov $L^{p,\omega}(D).$
sur l'espace vectoriel $\CM\{ z \}$ des séries convergentes en une variable. 

 %%%%%%%%%%%%%%%%
 %%%%%%%%%%%%%%%%%%%%%%%
  \section{  Morphismes born\'es}
 Nous allons a présent généraliser la notion d'opérateur différentiel. Commençons tout d'abord par étudier l'application qui à une série associe sa dérivée~:
 $$\CM\{ z \} \to \CM\{ z \},\ f \mapsto f' .$$
 Notons comme précédemment $D_t$ le disque centré en l'origine de rayon $z$ de rayon $t$. Je dis que cet opérateur est la limite directe d'une famille d'opérateur sur $C^\omega(D)$ et sur $L^{p,\omega}(D).$ Fixons $\l \in ]0,1[$, toute élément $f \in C^\omega(D)_t$ est holomorphe dans l'intérieur de $D_t$ donc dérivable, on a donc une famille de morphisme
 $$C^\omega(D)_t \to  C^\omega(D)_{\l t}$$
 pour tout $\l \in ]0,1[$ et tout $t \in ]0,S]$.  Soit, à présent, $E$ et $F$ des $S$-espaces de Kolmogorov.
\begin{definition} Une famille de morphismes $ u=(u_\l) \subset \Lt(E,F),\ \l \in ]0,1[$ entre deux espaces de Kolmogorov  est dite complète 
si pour tout $s<t \leq S $ et pour tout $\l \in ]0,1[$, on a un diagramme commutatif
$$\xymatrix{E_t \ar[r] \ar[d]_{u_\l} & E_s \ar[d]^{u_\l} \\
 F_{\l t} \ar[r] & F_{\l s}}$$
\end{definition}
Comme il ne peut pas y a voir d'ambiguïté, nous noterons
$$u:E_t \to F_s $$
la restriction de $u_\l$ à $E_t$ avec $s=\l t$. Nous dirons abusivement que $u$ est un morphisme complet d'espaces de Kolmogorov et nous noterons
$ u \in \Lt(E, F)$ au lieu de $ u \subset \Lt(E, F)$. Le foncteur limite directe associe à toute famille complète un unique morphisme.

Nous avons ainsi écrit la dérivation
$$\CM\{ z \} \to \CM\{ z \},\ f \mapsto f' $$
comme limite directe d'une famille complète
$$C^\omega(D) \to C^\omega(D),\ f \mapsto f'. $$
 Cette famille complète possède une propriété particulière~: elle satisfait l'inégalité de Cauchy. Rappelons ce résultat classique et élémentaire d'analyse complexe.
Considérons une fonction
$$f:D_t \to \CM$$
 holomorphe dans l'intérieur de $D_t$ et continue sur le bord de ce disque.  
En règle générale, la dérivée $f'$ n'est définie qu'à l'intérieur de $D_t$.  On peut donc avoir
 $$\lim_{z \to x} | f(z)|=+\infty $$
 lorsque $x$ est sur le bord du disque $D_t$. Cependant, on peut contrôler la vitesse de croissance de cette dérivée vers l'infini.
 En effet,  pour tout $z \in D_s$,  la formule de Cauchy donne~:
$$f'(z)=\frac{1}{2\sqrt{-1}\pi}\int_{\g_z} \frac{f'(\xi)}{\xi-z} dz, $$
où $\g_z$ le cercle de rayon $t-s$ centré en $z$. Après une intégration par parties, on obtient
$$f'(z)=\frac{1}{2\sqrt{-1}\pi}\int_{\g_z} \frac{f(\xi)}{(\xi-z)^2} dz. $$
Enfin, en paramétrant $\g_z$ par~:
$$\theta \mapsto z+(t-s)e^{2\sqrt{-1}\pi \theta} ,$$
on trouve~:
$$ f'(z)=\frac{1}{2\pi (t-s)}\int_{0}^{2\pi} \frac{f(\theta)}{e^{2\sqrt{-1}\pi \theta}} d\theta.$$
Ce qui nous donne {\em l'inégalité de Cauchy}~:
$$| f'(z)| \leq \frac{1}{t-s}| f |_t . $$
Donc lorsque $z$ s'approche du bord de $D_t$, le membre de gauche peut pas tendre vers l'infini plus vite que $1/(t-|z|)$. Pour un opérateur d'ordre $k$, on trouverait de la même façon une croissance en $1/(t-|z|)^k$. Cette propriété nous permet de généraliser la notion d'opérateur différentiel~:

\begin{definition}
\label{D::borne}
  Une morphisme complet $ u \in \Lt(E, F)$ est dit
    $k$-borné, $k \geq 0 $ s'il existe un réel $C>0$ tel que~:
  $$| u(x) |_t \leq \frac{Ce^2}{(t-s)^k} | x |_{s},\ {\rm pour\ tous\ }\ s<t \leq S, x\in E_t . $$
\end{definition} 
L'espace vectoriel des  morphismes  $k$-bornés
entre $E$ et $F$ sera noté $\Bt^k(E,F)$. La plus petite constante~$C$ vérifiant l'inégalité de la définition \ref{D::borne} définit une norme
$\| \cdot \|$ sur l'espace des morphismes $k$-bornés. La constante $e^2$ a été introduite pour simplifier les calculs qui viendront par la suite. Pour $k=0$, on prendra la norme usuelle c'est à dire
sans cette constante et avec $s=t$.
  
\begin{exemple} Les opérateurs différentiels ne sont pas les seuls opérateurs bornées. Comparons les espaces 
$L^{2,\omega}(D)$ et $C^\omega(D)$, définis précédemment, tous deux associés à l'espace des germes de fonctions en une variable.
Comme toute fonction continue sur le disque est intégrable, l'inclusion~:
$$C^\omega(D) \to L^{2,\omega}(D)  $$
est un morphisme $0$-borné. Ce qui est plus original c'est que ce morphisme est inversible d'inverse borné. En effet, un élément
$f \in L^{2,\omega}(D)_t$ est holomorphe dans tout disque de rayon inférieur à $t$, on a donc des inclusions naturelles
$$I_{ts}:L^{2,\omega}(D)_t \to  C^\omega(D)_s,\ s<t. $$
Vérifions que ces inclusions définissent un morphisme borné. Soit $z \in D_s$ et $\D_z$ un disque centré en $z$ de rayon $t-s$.
On a 
$$\int_{\D_z}| f|^2 \leq \int_{D_t}| f|^2=|f|_t  $$
pour tout $f \in L^{2,\omega}(D)_t$. Mais, par ailleurs, en écrivant la série de Taylor de $f$ au point $z$~:
$$f(z+h)=\sum_{i \geq 0}a_i h^i $$
on trouve
$$\int_{\D_z}| f|^2=A\sum_{j \geq 0} |a_j|^2 \, \frac{| h|^{2j+2}}{j+1}$$
où la constante $A$ est l'aire du disque $\D_z$. On a donc
$$\int_{\D_z}| f|^2 \geq | a_0|^2 (t-s)^2=|f(z)|^2(t-s)^2.  $$
Ce qui nous donne en définitive
$$|I(f)|_s \leq \frac{\sqrt{\pi t^2}}{(t-s)}|f|_t $$
Cette inclusion est donc un morphisme $1$-borné. La comparaison de différentes réalisations d'une même limite directe par des espaces de Kolmogorov équivalents, à opérateur borné près, joue un rôle important en théorie KAM.
\end{exemple}

%%%%%%%%%%%%%%%%%%%%%%%%%%%%%%%%%%%
\section{L'espace de Kolmogorov $\Bt^k(E,F)$.}
\begin{proposition}[J. F\'ejoz]Si $E,F$ sont des $S$-espaces de Kolmogorov alors l'espace vectoriel normé
$(\Bt^k_\tau(E,F),\|\cdot\|)$ est un espace de Banach.  
\end{proposition}
\begin{proof}
 Considérons un suite de Cauchy  $(u_n) \subset \Bt^k(E,F)$.
   Soit  $s<t \leq S $, les $(u_n)$ induisent, par
 restriction, des applications linéaires continues 
 $$v_n:E_t \to F_s.$$  
 Par définition de la norme $\| \cdot \|$, la suite $(v_n)$ est de Cauchy dans l'espace de Banach $\Lt(E_t,F_s)$ donc converge vers une application  linéaire. On a ainsi définit une limite $u$ de $(u_n)$ dans $\Lt(E,F)$. 
   
 Montrons à présent que $u$ est dans $ \Bt^k_\tau(E,F)$. L'inégalité
  $$|  \|u_n\| - \| u_m\| | \leq  \| u_n- u_m\| $$
   montre que  la suite $\| u_n\|$ est de Cauchy dans $\RM$ donc majorée par un constante~$C>0$.
   On a alors les inégalités~:
   $$| u(x) |_s \leq | u(x)-u_n(x) |_s+\frac{Ce^2}{(t-s)^k}| x |_t, \ {\rm\ pour\ tout\ } n, $$ 
  pour tout $x \in E_t$. Par conséquent, $ u$ est $k$-borné de norme au plus égale à $C$.  La proposition est démontrée.
\end{proof}  

\begin{corollaire} 
\label{C::Borel}
Lorsque $E,F$ sont des $S$-espaces de Kolmogorov, la somme directe des $S$-espaces vectoriels $\Bt^k(E[s],F[s])$, $s \in ]0,S]$
est munie d'une structure d'espace de Kolmogorov induite par les application de restriction
$$\Bt^k(E[t],F[t]) \to \Bt^k(E[s],F[s]),\ t>s. $$
\end{corollaire}
Cet espace sera appelé {\em l'espace des morphismes $k$-bornés entre $E$ et $F$}, on le note $\Bt^k(E,F)$. Il généralise la notion d'opérateur différentiel d'ordre $k$.

On a une inclusion d'espaces de Kolmogorov
$$\Bt^0(E,F) \subset \Bt^1(E,F) \subset \Bt^2(E,F) \subset \cdots $$
\begin{lemme} 
\label{L::rescale}
Soit $E,F$ des $S$-espaces de Kolmogorov avec $S \leq 1$. L'inclusion 
$$i:\Bt^k(E)\to \Bt^{k+n}(E)$$ est un morphisme borné de norme au plus $S^n$~:
$$ | i \circ u |_t \leq S^n | u |_t $$
pour tout $t \leq S$.
\end{lemme}
\begin{proof}
En effet:
$$| i \circ u(x)|_s \leq \frac{|u |_t}{e^2(t-s)^k}|x|_t=\frac{(t-s)^n|u |_t}{e^2(t-s)^{k+n}}|x|_t  \leq \frac{S^n |u |_t }{e^2(t-s)^{k+n}}|x|_t .$$
\end{proof}
%%%%%%
\section{Applications born\'es}
Dans un espace de Banach $E$, on peut définir les applications bornés dans un cadre linéaire, mais aussi pour les applications non linéaires: une application est bornée si l'image d'une boule est contenue dans une boule de rayon suffisamment grand. Cette notion se généralise aux espaces Kolmogorov~:
\begin{definition} Une application entre espaces de Kolmogorov
$$f:E \to F $$
est dite $k$-bornée par $C>0$ si pour tout $R \geq 1$, et tout $x \in E$, on a~:  
 $$ | x |_{t} \leq R(t-s)^k \implies | f(x)|_{s} \leq R C.$$
\end{definition}
Tout application linéaire $k$-borné est borné et plus généralement tout polynôme qui s'obtient à partir d'application bornés est lui-même borné.
\begin{lemme} 
\label{L::bounded_rescale}
Soit $E,F$ des $S$-espaces de Kolmogorov avec $S<1$.
une application
$$f:E \to F$$
$k$-bornée par $C$ est $(k+n)$-bornée $CS^n$.
\end{lemme}
\begin{proof}
Soit $x \in E$ tel que~:
 $$ | x |_{t} \leq R(t-s)^{k+n}.$$
 On a alors~:
  $$ | x |_{t} \leq RS^n(t-s)^k. $$
Comme $f$ est $k$-borné, on en déduit que~:
  $$  | f(x)|_s \leq R S^n C.$$ 
\end{proof}
%%%%%%%%%
 \section{R\'eéchelonnement d'un espace de Kolmogorov}
 Outre le fait que les espaces de Kolmogorov permettent d'interpréter de différentes façons des calculs sur les germes, ils mettent en relation la façon dont chaque quantité varie en fonction du paramètre qui définit le système direct.
  
Soit $E$ un espace de Kolmogorov. Le {\em rééchelonnement de $E$} par un facteur $\l >0$ est l'espace de Kolmogorov $E'$ définit par 
$$E'_s:=E[\l]_{\l s} $$
L'espace $E'$ est canoniquement isomorphe à $E$ et nous noterons par des apostrophes, l'image par cet isomorphisme canonique
ainsi que ceux induits sur les espaces d'applications linéaires qui lui sont associés.
 
\begin{proposition}\label{P::rescale} Soit $E,F$ des espaces de Kolmogorov. Les espace de Kolmogorov $E',F'$ obtenus après rééchelonnement par un facteur $\l$ possèdent les propriétés suivantes
\begin{enumerate}[{\rm i)}]
\item L'application canonique $\Bt^k(E[\l])' \to \Bt^k(E') $ est $0$-borné de norme au plus $\l ^{-k}$;
\item Si $\l \leq 1$ toute application $f:E \to F$ bornée par $C$ induit une application
$f':E'_\bullet \to F'_\bullet$ bornée par $\l^{-k} C$;
\end{enumerate}
\end{proposition}
\begin{proof}
Commençons par démontrer i). Si $u$ est un morphisme $k$-bornée alors 
$$|u'(x')|_s =| u(x)|_{\l s} \leq \frac{|u|_{\l t}}{e^2(\l t -\l s)^k}|x|_{\l t}=\frac{ |u|_{\l t}}{e^2\l^k(t-s)^k}|x'|_{ t}. $$
Par conséquent~:
$$|u|_{\l t}= \l^k |u'|_t .$$
La démonstration de ii) est identique.  Pour tout $R \geq 1$ et  tout $x \in E$, on a~:  
$$ | x' |_t  \leq R( t- s)^k \implies | x |_{\l t} \leq \l^{-k}R( \l t- \l s)^k \implies | f(x)|_{\l s}\leq \l^{-k}R C.$$
Ce qui démontre la proposition
\end{proof}

Nous dirons qu'une application $k$-bornée est un projecteur si c'est un idempotent de norme au plus $1$. Comme dans l'étude des espaces de Hilbert, les projections jouent un rôle significatif dans la théorie des espaces de Kolmogorov.
\begin{corollaire} 
\label{C::prescale} Toute projection $k$-bornée de $E$ induit une projection $2k$-bornée de $E'$.
\end{corollaire}
\begin{proof}
D'après le lemme~\ref{L::rescale}, l'inclusion
$$\Bt^k(E[\l]) \to \Bt^{2k}(E[\l])   $$
est $0$-bornée et sa  norme est au plus $\l^k$. Or d'après la Proposition~\ref{P::rescale}  la norme de l'application canonique
$$ \Bt^{2k}(E[\l])'  \to  \Bt^{2k}(E')$$
est au plus $\l^{-k}$. Donc l'application canonique
$$\Bt^k(E[\l])' \to \Bt^{2k}(E')   $$
 est $0$-bornée de norme au plus un. En particulier, elle envoie un projecteur sur un projecteur.
  \end{proof}

\chapter{Th\'eor\`emes de point fixe}
%%%%%%%%%%%
Nous pouvons maintenant généraliser les deux théorèmes de points fixes, qui sont à la base de l'algorithme de Kolmogorov abstrait, à la dimension infinie.
\section{Calcul fonctionnel dans un espace de Kolmogorov}
L'espace des opérateurs bornés d'un espace de Banach $\Lt(E)$ forme une algèbre de Banach pour la norme d'opérateur~:
$$\| u v \| \leq \| u \| \, \| v \| . $$
Il en résulte que pour toute série analytique $f \in \CM\{ z \}$, 
$$u \mapsto f(u)$$ est bien définie dans un voisinage de l'origine. Nous souhaitons étendre ce résultat aux espaces de Kolmogorov.

Si $u,v$ sont des morphismes,  respectivement $k$ et $k'$ borné, alors leur composition $u  v$ est
  $(k+k')$-borné et on a l'inégalité
  $$| u v|_t \leq 2^{k+k'}|u|_t|v|_t .$$
  En effet, comme $u$ est $k$-borné, on a, pour  tout $x\in E$~:
 $$| (u v)(x) |_s \leq  \frac{2^k|u|_{s+(t-s)/2}}{e^k(t-s)^{k}} |v (x)|_{s+(t-s)/2} $$
 et, comme $v$ est $k'$-borné, on a de plus~:
 $$ |v (x)|_{s+(t-s)/2}  \leq   \frac{2^{k'} |v|_t }{e^{k'}(t-s)^{k'}} |x|_{t}. $$
 Par ailleurs $|u|_{s+(t-s)/2}  \leq |u|_{t}$, par conséquent~:
 $$ | (u v)(x) |_s \leq 2^{k+k'}|u|_t|v|_t |x|_{t}. $$ 
  Ce qui démontre l'affirmation.
  
 On considère la transformation de Borel~:
 $$B:\CM\{ z \} \to \CM\{ z\},\ \sum_n a_n z^n \mapsto \sum_n \frac{a_n}{n!} z^n . $$
 Nous notons $\RM_+\{ z\} \subset \CM\{ z\}$, le sous-espace des séries à coefficients réels positifs ou nuls.
 \begin{proposition}
 \label{P::Borel} Soit $E$ un espace de Kolmogorov et $u$ un morphisme $k$-borné de $E$ tel que
   $$\nu:=\frac{| u |_t}{t-s} < R(f) ,$$
alors
$$| Bf(u) x |_s \leq f(\nu)   |x |_t$$
pour tout $x \in E_t$. 
 \end{proposition}
 \begin{proof}
 En effet, en découpant l'intervalle $[s,t]$ en $n$ parties égales et en utilisant le fait que $ | u |_{s'} \leq | u |_t$ pour tout $s'\leq t$, on obtient~:
     $$|u^n(x)|_s \leq  \frac{n}{e^2(t-s)}| u |_{s+(t-s)/n}  |u^{n-1}(x)|_{s+(t-s)/n} \leq \dots \leq  \frac{n^n}{e^{2n}(t-s)^n}  | u |_t^n |x|_{t}  . $$
 Or
$$ \frac{n^n}{e^n} \leq  n!\,,$$
donc~:
$$\frac{n^n  |u|_t^n}{e^{2n}(t-s)^n}   \leq  n!  \left( \frac{ | u |_t}{e(t-s)}\right)^n$$
et~::
$$ | u^n(x) |_{s}   \leq  n! \left( \frac{ | u |_t}{e(t-s)}\right)^n  | x |_t.$$
Ce qui nous donne bien~:
$$|   \sum_{n \geq 0} \frac{a_n}{n!} u^n( x) |_{s} \leq \sum_{n \geq 0}  \frac{a_n}{n!} | u^n(x) |_{s} \leq \sum_{n \geq 0} a_n \left(\frac{|u|_t}{t-s} \right)^n | x |_t=f(\nu)   |x |_t.$$
 \end{proof}  
 
Appliquons la proposition précédente à la série~:
$$e^z=B\left(\frac{1}{1-z}\right) \in \CM\{ z \}.$$ 
Nous obtenons que
$$| e^u x |_{s} \leq \frac{1}{1-\nu }   |x |_t,$$
pourvu que
$$  \nu=\frac{| u |_t}{t-s} <1.$$
Nous avons ainsi généralisé la correspondance entre algèbre de Lie et groupe de Lie aux espace de Kolmogorov~:
\begin{corollaire} Soit $u \in \Bt(E)^{(1)} $ un morphisme $1$-borné tel que
$$\forall t \leq S,\ | u |_t<t. $$
Pour toute fonction croissante
$$\phi:]0,S] \to ]0,S] $$
telle que
$$\forall t \leq S,\ \phi(t)<t- | u |_t $$ la série $e^u$ définit un morphisme qui envoie $E_t$ sur $E_s$ avec $ s~=~\phi(t)$.
\end{corollaire}

Le corrolaire suivant sera pratique pour montrer que l'algorithme de Kolmogorov est à convergence rapide~:
 \begin{corollaire} 
 \label{C::attracteur}
 Si une série $f \in \RM_+\{ z\}$ appartient à la puissance $k$\up{ième} de l'idéal maximal~:
$$f=z^kg(z),\ g \in \RM_+\{ z\}, $$
alors pour tout $r $ inférieur au rayon de convergence de $f$, on a:
$$| Bf(u) x |_s \leq g(r) \nu^k  |x |_t .$$
\end{corollaire} 

 \section{Premier th\'eor\`eme de point fixe}
 Comme nous l'avons vu en dimension finie, la convergence du procédé itératif de Kolmogorov découle d'un théorème de point fixe. Dans le cas où la transversale est réduite à $\{0\}$, l'énoncé   est particulièrement simplifié~:
  \begin{theoreme}
 \label{T::point_fixe} Soit $E$ un $S$-espace de Kolmogorov, $f \in \Mt^2 \subset \CM\{ z \}$ une série convergente et
 $$ j:E \to E$$ un morphisme borné. Pour tout $t \in ]0,S]$, il existe un voisinage de l'origine $B_t\subset E_t$ tel que pour tout $u_0 \in B_t$ et tout $q \in ]1,2[$, $\rho \in ]0,1[$, la suite
 $$u_{n+1}=j \circ Bf(u_n) $$
 vérifie
 $$| u_n |_s =O(\rho^{q^n}) $$
 pour $s$ suffisamment petit.
 \end{theoreme}
  \begin{proof}
Si le théorème est démontré pour $s$ assez grand en rééchelonnant notre espace, cela entraîne le théorème pour tout $s$.
  
  Soit $k$ tel que $j$ soit $k$-borné.  Comme $f$ possède un point critique à l'origine, d'après le corollaire~\ref{C::Borel}, il existe $C,r$ tel que
$$(*)\ |jBf(u)|_s \leq C(t-s)^{-k}| u|_t^2$$
pourvu que $| u|_t/(t-s) \leq r$.  
Soit  $\rho \leq r$ et $l>0$ tels que~:
$$ 2^{k(n+1)/l}\rho^{(2-q)q^n} \leq  1.$$

 Définissons la suite $(s_n)$ par
$$ s_{n+1}=s_n-\frac{1}{2^{n/l}},\ s_0=t.$$
Nous allons démontrer le théorème pour
$$s >\sum_{n \geq 0} \frac{1}{2^{n/l}}.$$
Cette inégalité nous permet de garantir que $\lim s_n>0$.
     
Quitte à multiplier les normes par une même constante,on peut supposer que
$$C=1.$$
Nous définissons le voisinage de l'origine par la condition
 $$  | u_0|_{s_0}^2 \leq  \rho.$$
 Montrons par récurrence que
 $$  | u_n|_{s_n}^2 \leq  \rho^{q^n}.$$
 On a~: 
 $$u_{n+1}=j\circ Bf(u_n) $$
 donc d'après (*)~:
 $$| u_{n+1}|_{s_{n+1}} \leq 2^{k(n+1)/l}| u_{n}|_{s_n}^2 $$
 et, par hypothèse de récurrence, on a 
 $$ | u_{n}|_{s_n}^2 \leq 2^{k(n+1)/l} \rho^{2 q^n}= 2^{k(n+1)/l} \rho^{(2-q)q^n}\rho^{q^{n+1}}. $$
Or
 $$2^{k(n+1)/l} \rho^{(2-q)q^n} \leq 1,  $$
 ce qui démontre le théorème.
 \end{proof}

\section{Deuxi\`eme th\'eor\`eme de point fixe}
 \begin{theoreme} 
 \label{T::point_fixe_2}
 Soit $E,F$ des $S$-espaces de Kolmogorov et $f_1,f_2 \in \CM\{ z,w \}$ deux séries analytiques telles que 
 $$f_i(z,0)=\d_w f_i(z,0)=0$$
 et
  $$j:F \times E \to F \times E $$
 une application bornée. Posons
$$G: F \times E\to F \times E ,\ (a,x) \mapsto j(a+f_1(a,x),f_2(a,x))  $$
Pour tous $a \in F$, $\rho<1$, $\e>0$, $t \in ]0,S]$ il existe un voisinage $B_t$ de l'origine dans $E_t$ tel que pour tout $x \in B$
La suite $(a_n,x_n)=G^n(a,x)$ est convergente et~:
$$| x_n |_s =O( \rho^{2^{n-\e}}) $$
pour tout $s$ suffisamment petit.
\end{theoreme}
\begin{proof}
Comme précédemment, il suffit de démontrer le théorème pour $s$ assez grand.
  
  Soit $k$ tel que $j$ soit $k$-borné.

 D'après le corollaire~\ref{C::Borel}, il existe $C,R,r$ tel que
$$(*)\ |jBf_1(b,u)|_s+|jBf_2(b,u)|_s \leq C(t-s)^{-k}| u|_t^2$$
pourvu que $| u|_t/(t-s) \leq \rho $ et $| b|_t \leq R$. De plus quitte à multiplier les normes par une constante et à remplacer $\rho$ par un réel qui lui est inférieur, on peut choisir $l$ pour que
\begin{align*}
\rho<R;\\
\sum_{i=0}^n \rho^{2^i} \leq \frac{1}{2} ;\\
 C=1;\\
 2^{k(n+1)/l}\rho^{(2-q)q^n} \leq  1
\end{align*}

Définissons  la suite $(s_n)$ par
$$ s_{n+1}=s_n-\frac{1}{2^{n/l}},\ s_0:=t.$$
Nous allons démontrer le théorème pour
$$t >\sum_{n \geq 0} \frac{1}{2^{n/l}}.$$

 Nous définissons le voisinage de l'origine par la condition
 $$  | x_0|_{s_0} \leq  \rho,\ | a |_{s_0} \leq \frac{R}{2}.$$
 Montrons par récurrence que
 $$  | x_n|_{s_n} \leq  \rho^{q^n},\ | a_n |_{s_n} \leq \frac{R}{2}+ R \sum_{i=0}^n \rho^{2^i}.$$
 Comme dans le cas homogène~: 
 $$| j\circ Bf_i(x_n) |_{s_{n+1}} \leq  \rho^{q^{n+1}}. $$
 et par suite
 $$   | x_{n+1}|_{s_n} \leq  \rho^{q^{n+1}},\  | a_{n+1} |_{s_{n+1}} \leq \frac{R}{2}+ R \sum_{i=0}^{n+1} \rho^{2^i}<R  $$
 Le théorème est démontré.
 \end{proof}

 \chapter{Le th\'eor\`eme de Kolmogorov abstrait}

%%%%%%%%%%%%%%%%%%%%%%%%%%%%%%%%%%%%%%%%%%%%%%%%%%%%
\section{Th\'eor\`eme des exponentielles}
L'algorithme de Kolmogorov fait intervenir un produit infini d'exponentielle, le théorème des exponentielles
 nous permet de garantir que ce produit infini
converge.
\begin{theoreme} 
\label{T::convergence}
Soit $E$ un espace  un espace de Kolmogorov et $(u_n) \subset \Bt^1_\tau(E)$ une suite de $\tau$-morphismes $1$-bornés.
Supposons que pour tout $s$, on ait 
$$\sum_{i \geq 0} | u_i|_s < s  $$
 alors la suite $(g_n)$ définie par
$$g_n:=e^{u_n}e^{u_{n-1}}\cdots e^{u_0}  $$
converge vers un élément inversible de $\Lt(E)$.
 \end{theoreme}
 \begin{proof}
La transformée de Borel définie pour les séries en une variable se généralise aisément au cas de $n$ variables~:
$$B:\CM\{ z_1,\dots,z_n \} \to \CM\{ z_1,\dots,z_n\},\ \sum_I a_I z^I \mapsto \sum_I \frac{a_I}{|I|!} z^I  $$
où $I=(i_1,\dots,i_n)$ est un multi-indice et $|I|=i_1+i_2+\dots+i_n$.
 \begin{proposition} Soit $E$ un espace de Kolmogorov et $f_1,\dots,f_n \in \RM_+\{ z\}$ des séries convergentes d'une variable. Notons
 $g_i=Bf_i$ la transformée de Borel de $f_i$. Pour tout $ \l \in \min_i( ]0, 1-\frac{| u_i |_s}{s}[$), l'application
 $$\left(\Bt^1(E)\right)^{n+1} \to \Lt(E),\ (u_0,\dots,u_n) \mapsto g_0(u_0)g_1(u_1) \dots g_n(u_n) $$
  qui envoie $E_s$ sur $E_{\l s}$
 est bien définie et
 $$| \prod_{i=0}^n g_i(u_i) x |_{s}\leq B^{-1} \left(\prod_{i=0}^n g_i\left(\frac{| u_i |_t}{t-s}\right) \right) |x|_{t}. $$
 \end{proposition}
 \begin{proof} 
Notons $T_{j,n} \subset \ZM^j$ l'ensemble des éléments $ I=(i_1,\dots,i_j)$ dont les coordonnées sont dans
 $\{0,\dots,n \}$ qui appartiennent au simplexe $x_{i+1} \leq x_i$ pour $i=1,\dots,j$. 
   
Pour tout $I=(i_1,\dots,i_j) \in T_{j,n}$ on note $\s(I)$ le vecteur dont les composantes sont obtenues à partir de $I$ par permutation pour que
$\s(I)_p \geq \s(I)_{p+1}$. 

Posons
\begin{enumerate}[{\rm 1)}]
\item$g_n(z)=\sum_k a_k^n z^k$ ;
\item $u[I]:= u_{\s(I)_1} u_{\s(I)_2} \cdots u_{\s(I)_j},\ I \in C_{j,n} $ ;
\item $a[I]:=a_{i_1}^1a_{i_2}^2\dots a_{i_n}^n $
\end{enumerate}

Regroupons les termes du produit de la façon suivante~:

\begin{align*}
g_0(u_0)g_1(u_1) \dots g_n(u_n) &=  \sum_{j \geq 0} \sum_{I \in T_{j,n}} a[I] u[I]\\ & = a_0+\sum_{i=0}^n a^i_1 u_i+\sum_{i=0}^n a^i_2 u_i^2
+\sum_{j=0}^n \sum_{i=j+1}^n  a^i_1 a^j_1 u_i u_j +\dots .
 \end{align*}

Posons 
$$z_{I,s}:=|u_{i_1}|_s |u_{i_2}|_s \cdots |u_{i_n}|_s.$$
En raisonnant comme dans la démonstration de la proposition précédente, on obtient l'inégalité~:
$$\frac{1}{j!} | u[I]|_s  \leq \prod_{p=0}^j | u_{i_p}|_s=  z_{I,s}$$
et par suite
$$\left| u[I] (x)\right|_{\l s} \leq  \a^{-j} z_{I,s}  | x|_s,\ \forall \l \in ]0,1[. $$
avec
$$ \a=\frac{1}{(1-\l)s},$$
On obtient ainsi l'estimation
$$| g_n x |_{\l s} \leq  \left( \sum_{j \geq 0}j!\a^j \sum_{I \in T_{j,n}} a[I] z_{I,s} \right)| x|_s=B^{-1}(f(\a z_{1,s})f(\a z_{2,s}) \dots f(\a z_{n,s}))| x|_s  .  $$
Ceci démontre la proposition.
 \end{proof}
 \begin{exemple}Prenons $g_0=\dots=g_n=e^z$, on a alors
$$\prod_{i=0}^n g_i(z_i)=e^{z_1+\dots+z_n} $$
et
$$ B^{-1} \left(\prod_{i=0}^n g_i\right)=\frac{1}{1-(z_1+\dots+z_n)} $$
donc la proposition nous donne l'estimation
$$| \prod_{i=0}^n e^{u_i} x |_{s}\leq  \frac{1}{1-(z_1+\dots+z_n)}  |x|_t $$
avec
$$z_i=\frac{| u_i |_t}{t-s}. $$
\end{exemple}

 Achevons la démonstration du théorème. Pour cela, fixons $s$.  
Par hypothèse, on a
$$ \sum_{i \geq 0} |u_i|_s<s .$$
On peut donc choisir $\l \in]0,1[$ tel que
$$ \sum_{i \geq 0} |u_i|_s<(1-\l)s.  $$

 Notons $\| \cdot \|_{\l}$ la norme d'opérateur dans $\Lt(E_s,E_{\l s})$. 
La proposition précédent donne l'estimation
$$ \| g_n \|_\l \leq    \frac{1}{1-\frac{3}{(1-\l)s}\sum_{i \geq 0} |u_i|_s}.$$
La suite $(g_n)$ définit donc, par restriction, une suite uniformément bornée d'opérateurs dans $\Lt(E_s,E_{\l s})$. 

Soit à présent $\mu \in ]0,1[$ tel que
 $$\sup_{i \geq 0}  |u_i|_{\l s} <(1-\mu)\l s .$$ 
 Nous allons montrer que la suite $(g_n)$ définit, par restriction, une suite de Cauchy dans $\Lt(E_s,E_{\mu \l s})$. La proposition en découlera, car ce dernier est un espace de Banach pour la norme d'opérateur. 

Je dis que la série de terme gé\-né\-ral $ \| g_n-g_{n-1}\|_{\l \mu }$ est convergente. Pour le voir, écrivons
$$g_n-g_{n-1}=(e^{u_n}-\Id)g_{n-1} $$
où $\Id \in \Lt(E)$ désigne l'application identité.

La fonction $e^z-z$ est la transformée de Borel de $z/(1-z)$ donc, d'après la Proposition~\ref{P::Borel}~: 
$$ | (e^{ u_n}-\Id) y |_{\l \mu s} \leq  
\frac{  \nu_n}{1-\nu_n} | y |_{\l s} ,$$
avec
$$\nu_n:=\frac{|u_n|_{\l s}}{(1-\mu)\l s}. $$
pour tout $y \in E_{\l s}$.

 En prenant $y=g_{n-1} x$, ceci nous donne l'estimation
$$ \| (e^{ u_n}-\Id) g_{n-1} \|_{\l \mu}  \leq \frac{  \nu_n}{(1-\nu_n)(1-\frac{1}{(1-\l)s}\sum_{i \geq 0} |u_i|_s)}$$
On  a donc
$$ \| (e^{ u_n}-\Id) g_{n-1} \|_{\l \mu}=O(\nu_n)=O(|u_n|_{\l s}).$$
La suite réelle $\| g_n-g_{n-1} \|_{\l \mu} $ est par conséquent sommable. Ceci montre que la suite $(g_n)$ converge vers un élément
$g \in \Lt(E_s,E_{\l \mu s})$ pour tout $s$. 

On démontre de même que la suite $(h_n)$  définie par
$$h_n=e^{- u_0}e^{- u_1} \cdots e^{-u_n}$$ 
converge vers un élément $h \in \Lt(E)$. Pour tout $n \in \NM$, on a~:
$$g_nh_n=h_ng_n=\Id$$
 donc $gh=hg=\Id$. Ce qui montre que $h$ est l'inverse de $g$.
 Le théorème est démontré. 
 \end{proof}
 \section{Th\'eor\`eme de Kolmogorov abstrait}

Nous obtenons notre premier résultat important~:
  \begin{theoreme}
 \label{T::groupes} Soit $E$ un espace de Kolmogorov, $a \in E$, $M$ un sous-espace de Kolmogorov de $E$, $\alg$ un sous-espace vectoriel de $\Bt^1(E)^{(2)}$ tel que l'application
  $$\rho:\alg \to M,\ u \mapsto u(a)$$
  soit bien définie. Soit $G$ un sous-groupe fermé de $\Lt(E)$ contenant $\exp(\alg)$. Si $\rho$
  possède un inverse à droite borné alors l'orbite de $a$ sous l'action de $G$ est égale à $a+M$.
 \end{theoreme}
 \begin{proof} 
 Notons 
 $$j:E \mapsto \alg$$
   l'inverse de l'application $\rho$. L'algorithme de Kolmogorov est donné par~:
\begin{enumerate}[1)]
\item $b_{n+1}:=e^{-u_n}(a+b_n)-a$ ;
\item  $u_{n+1}:=j(b_{n+1}).$
\end{enumerate}
avec $ b_0=b,\ u_0=j(b) $. Comme les $u_n \in \Bt^1(E)^{(2)} $, ils sont exponentiables.

Dans cette itération, la suite $(u_n)$ est définie par la formule
$$u_{n+1}=j(e^{-u_n}(\Id+u_n)(a)). $$
La fonction
$$x \mapsto e^{-x}(\Id+x)  $$
est la transformée de Borel d'une fonction avec un point critique à l'origine et $j$ est $k$-borné pour un certain $k$. Comme l'inclusion
$$\psi:\Bt^k(M,\alg)[t] \to \Bt^{k+2}(M,\alg)[t] $$
a une norme qui tend vers $0$ en  $t^2$. L'espace de Kolmogorov rééchelonné par $t$~:
$$ E'_s=E_{st}$$
induit un morphisme 
$$\p:\Bt^1(E) \to \Bt^1(E') $$
de norme majoré par $t^{-1}$. Donc la norme $|\p(u_n)|_t$ tend vers $0$ avec $t$. Par conséquent, en appliquant le théorème du point fixe (Théorème~\ref{T::point_fixe}) avec $t$ suffisamment petit, on trouve $\rho<1$ tel que~:
$$|\p(u_n)|_s =o(\rho^{2^n}) $$
et par suite~:
$$|u_n|_s =o(\rho^{2^n}s)  $$
pour $s$ assez petit.
D'après le Théorème de Convergence (Théorème~\ref{T::convergence}), la suite formée par les produits
$$e^{u_n} \dots e^{u_1}  e^{u_0} $$
converge. Ceci   achève la démonstration du théorème.
 \end{proof}
La généralisation au cas non-homogène est immédiate~:
    \begin{theoreme}
 \label{T::groupes} Soit $E$ un espace de Kolmogorov, $G$ un sous-groupe fermé de $\Lt(E)$ et $\alg$ un sous-espace vectoriel de $\Bt^1(E)^{(2)}$ tel que $e^\alg \subset G$. Soit $a \in E$, $F \subset M$  des sous-espaces de Kolmogorov tels que l'application  
 $$\rho:F \times \alg \to F \times M/F,\ (\a,u) \mapsto (\a,u (a+\a))$$
 soit bien définie et possède un inverse à droite borné
  $$j:F \times M/F \to F \times \alg,\ (\a,b) \mapsto (\a,j(\a,b)).$$
   L'orbite de $a+F$ sous l'action de $G$ est alors égale à $a+M$.
 \end{theoreme}

   %%%%%%%%%%%%%%%
 \chapter{Espaces de Kolmogorov en g\'eom\'etrie analytique}
 \label{S::analytique}
 %%%%%%%%%%%%%
 \section{Germes le long d'un compact}
 Nous avons vu que l'espace vectoriel des séries convergentes $\CM\{ z\}$ se réalise comme la limite directe de différents espaces
 de Kolmogorov $C^{0,\omega}$, $L^{p,\omega}$. Nous allons généraliser cette construction afin de pouvoir l'appliquer au cas des tores invariants. 
  
 Soit $X$ une variété complexe non nécessairement lisse et $K \subset X$ un sous-ensemble compact.
 Considérons l'espace vectoriel 
 $$\Ot_{X,K}=\underrightarrow{\lim}\, \Ot_X(U) $$
 où $U$ parcourt l'ensemble des ouverts contenant $K$ ordonné par l'inclusion. On muni cet espace de la topologie la plus fine qui rende les applications
 $$\Ot_X(U) \to \Ot_{X,K} $$
   continue.
   
 Un élément de $\Ot_{X,K}$ est une section du faisceau $\Ft$ au voisinage de $K$, pour laquelle on oublie de préciser la taille du voisinage de $K$ sur laquelle  elle est définie.  Prendre la limite directe revient donc à identifier deux sections qui sont égales sur un ouvert contenant $K$~:
 $$f \sim g \iff \exists U \supset K,\ f_{\mid U}=g_{\mid U}.  $$ 
 Dans le cas où $K$ est réduit à un point, on retrouve la notion de germe en un point. Nous parlerons donc, comme nous l'avons fait pour l'énoncé du théorème des tores invariants, de germes de fonctions holomorphes en un compact.  
 
 Choisissons par exemple 
 $$K =\{ z \in  \CM^*:|z|=1 \}$$
  le cercle unité. Un élément de $\Ot_{X,K}$ est une série
  $$\sum_{n \in \ZM} a_nz^n$$ analytique dans une certaine couronne
 $$C_t:=\{ z \in \CM: 1-t < | z | <1+t\}. $$
 L'analyticité de la série s'exprime par la décroissance exponentielle des coefficients et, en posant $z=r e^{2n\pi\sqrt{-1}}$, on voit qu'une telle série est le prolongement analytique de la série de Fourier
 $$f_{\mid K}(\theta)=\sum_{n \in \ZM} a_ne^{n\sqrt{-1}\theta} $$

 %%%%%%%%%%%%%%%%%%%%%%%
 \section{Les foncteurs  $C^{\omega}$ et $L^{2,\omega}$}
\label{SS::Banach}

Soit $K=(K_s),\  s \in ]0,S[$ une famille croissante  de compacts  d'un espace analytique $X$.  Les espaces de Banach~:
  $$C^\omega(K)_s:=C^0(K_s,\CM) \cap \Ot_X(\mathring{K_s}) $$
définissent un espace de Kolmogorov pour les normes
  $$| f |_s:=\sup_{z \in K_s} |f(z)| $$
 On a ainsi définit un foncteur de la catégorie des familles croissantes de compacts dans celle des espaces de Kolmogorov
 $$C^\omega:Comp \to EK,\ K \mapsto C^\omega(K).$$ 
Prenons à présent $X=\CM^n$. Nous dirons que $K$ {\em vérifie la propriété de Cauchy} si le polydisque de centre $x$ et de rayon $(t-s)$ est contenu dans $K_t$.
On peut alors généraliser les inégalités de Cauchy dans ce contexte~:
\begin{proposition} Si $K=(K_s),\  s \in ]0,S[$ une famille croissante  de compacts de $\CM^n$ vérifiant la propriété de Cauchy alors tout opérateur différentiel d'ordre $k$ définit une application $k$-bornée de $C^\omega(K)$.
\end{proposition}
La démonstration est analogue à celle du cas $n=1$, $K=\{0 \}$.
   
    %%%%%%%%%%%%%%%%%
%%%%%%%%%%%%%%%%%%%%%%%%%%%%%%%

Supposons $X$  munit d'une forme volume $dV$ et $K=(K_s), \ s \in ]0,S[$ une famille  strictement croissante de compacts de Stein de $X$. 
Les espaces de Banach
$$ L^{p,\omega}(K)_s:=L^p(K_s,\CM) \cap \Ot_X(K_s)$$
définissent des espaces de Kolmogorov pour la norme
$$| f |_s:= \int_{K_s} |f(z)|^2p dV.  $$
On a ainsi définit des nouveaux foncteurs de la catégorie des familles croissante de compacts dans celle des espaces de Kolmogorov
 $$L^{p,\omega}:Comp \to EK,\ K \mapsto L^{p,\omega}(K).$$
 
 \section{Homotopie de foncteurs à valeur dans $EK$}
 Comme toute fonction continue est intégrable on a un diagramme commutatif
$$\xymatrix{Comp \ar[r]^{C^\omega} \ar[rd]_{L^{2,\omega}}& EK \ar[d] \\
& EK}$$
Inversement si $f \in L^{2,\omega}_t$ alors $f$ est holomorphe à l'intérieur de $K_t$ donc continue sur $K_s \subset K_t$ pour tout $s<t$.
On a ainsi une famille complète de morphismes 
$$J: L^{2,\omega}(K)_t \to C^\omega(K)_s.$$
Lorsque la limite directe de $C^\omega(K)$ et $L^{2,\omega}(K)$ est l'espace des germes le long des $\cap_s K_s$, l'application $J$ devient
l'identité sur la limite directe

\begin{proposition} Si $K$ est une famille de compacts de $\CM^n$ muni de la forme volume 
$$dV=2(2i\pi)^{-n}\prod_{i=1}^nd\bar z_i \w d z_i$$ qui vérifient la propriété de Cauchy alors $J$ est un morphisme $1$-borné de norme au plus $1$.
\end{proposition}
\begin{proof}
Prenons $s<t$ et   $f \in L^{2,\omega}(K)_t$. La série de Taylor de $f$ au point   $w \in K_s$   donne~:
$$f(z)=\sum_{j \geq 0} a_j (z-w)^j,\ a_j \in \CM. $$
Comme $K$ vérifie la propriété de Cauchy, le polydisque   $\G_w$ centré en $w$ de rayon $\s=t-s$ est contenu dans $K_t$. Par ailleurs
$$ \int_{\G_w} | f(z)|^2 dV=\sum_{j \geq 0} |a_j|^2 \frac{\s^{2j+2}}{j+1} =|a_0|^2 \s^{2} +\dots $$
et comme $\G_w \subset K_t$, on en déduit que~:
$$ \int_{\G_w} | f(z)|^2 dV \leq   \int_{K_t} | f(z)|^2 dV=| f |_t^2  $$
Ceci montre que
$$ |f(w) |=|a_0| \leq  \frac{1}{\s}\left( \int_{\G_w} | f(z)|^2 dV \right)^{1/2} \leq  \s^{-1} | f |_{s+\s}$$
pour tout $w \in K_s$ et démontre la proposition.
 \end{proof}
 Nous avons donc deux foncteurs $C^\omega$ et $L^{2,\omega}$ avec des applications bornées entre eux. Nous dirons alors que ces deux foncteurs sont {\em homotopes}.
 %%%%%%%%%%%
 
 Notons que le cas analytique réel ne pose pas de difficulté: si la variété $X $ munie d'une involution anti-holomorphe 
$$\tau:X \to X $$
et $K$ une famille de compacts contenu dans $X_{\RM }$. Une fonction $f:X \to \CM$ est dite réelle si
$$f(\tau z)=\overline{f(z)} .$$
Les sous-espaces de fonctions réelles de $C^\omega(K)$ et $L^{2,\omega}(K)$ définissent des sous-espaces de Kolmogorov.  

%%%%%%%%%%%%%%%%%%%%%%%%%%%%
 \section{S\'eries de Fourier}
 Revenons aux fonctions analytique définit sur un voisinage du cercle unité dans $X=\CM^*$. Notons $K$ la famille croissante de voisinages du cercle unité
$$K_s=\{ z \in \CM^*: 1-s \leq | z| \leq 1+s \} $$
et munissons $X$ de la forme volume
$$dV=\frac{1}{4i \pi} d\bar z \w dz. $$
En posant $z=re^{i\theta}$, la condition $z \in K_s$ devient
$$ 1-s \leq r \leq 1+s.$$
Le produit hermitien 
$$\< z^n | z^m \> =\frac{1}{4i \pi}\int_{K_s}z^n \bar z^m d\bar z \w dz=\frac{1}{2\pi}\int_{r=1-s}^{r=1+s}\int_{\theta=0}^{\theta=2\pi}r^{n+m+1}e^{i(n-m)\theta}d\theta dr $$
est nul pour $n \neq m$. 
Par conséquent, ces fonctions sont orthogonales et
\begin{align*}\< z^n | z^n \> &= \int_{r=1-s}^{r=1+s} r^{2n+1} dr \\
&=\frac{1}{2n+2}\left((1+s)^{2n+2}-(1-s)^{2n+2}\right).
\end{align*}
pour $n\neq -1$.
Par ailleurs
 $$ (1+s)^{2n+2}-(1-s)^{2n+2} \sim \left\{ \begin{matrix}
  (1+s)^{2n+2}&pour& n \to +\infty \\
  -(1-s)^{2n+2}&pour & n \to -\infty
  \end{matrix}\right. $$
 Les fonctions de $L^{2,\omega}(K)_s$ s'écrivent donc comme des séries de Fourier
$$f(z)= \sum_{n \in \ZM} a_n z^n$$
telles que
$$\sum_{n >0} \frac{|a_n|^2}{2n+2}(1+s)^{2n+2}  <+\infty$$
et
$$\sum_{n < 1} \frac{|a_n|^2}{|2n+2|}(1-s)^{2n+2}  <+\infty.$$
On retrouve bien la correspondance usuelle entre analyticité et décroissance exponentielle des coefficients de Fourier.  En effet, pour $n \to +\infty$, on a~:
$$(1+s)^{n+1}=e^{(n+1)\log (1+s)} \sim e^{n s}. $$
Le terme général de la série $$\sum_{n >0} \frac{|a_n|^2}{2n+2}(1+s)^{2n+2}  <+\infty$$
tend vers zéro donc
$$a_n=o(n e^{-n s}) .$$
De même, lorsque $n \to -\infty$, les coefficients de la série sont à décroissance exponentielle.
%%%%%%%%%%%%%%%  
%%%%% 
 
%%%%%%%%%%
 %%%%%%%%%%%%%%%%%%%%%%%%%%%%%%%%%%%%%%%%%%%%
   \chapter{Le th\'eor\`eme KAM singulier}
   Nous allons commencer par une version singulière du théorème KAM. La démonstration est essentiellement la même que celle du théorème
   des tores invariants, mais l'absence de paramètre perturbatif permet de simplifier les notations.
      \section{\'Enoncé du th\'eor\`eme}
   Considérons l'algèbre locale 
   $$\CM\{ q,p \}:=\CM\{ q_1,\dots,q_n,p_1,\dots,p_n \}$$
   des séries convergentes dans les variables $q_i,p_i$
   munie de sa structure symplectique usuelle. Nous notons $\Mt$ l'idéal maximal des séries qui s'annulent à l'origine.
   \begin{theoreme}  Soit
 $$H=\sum_{i=1}^n \omega_i p_iq_i$$
 et $I \subset \CM\{ q,p \} $ l'idéal engendré par les $p_iq_i$.
 Si le vecteur $\omega=(\omega_1,\dots,\omega_n)$ est diophantien alors pour tout $R \in \Mt^3$, il existe un automorphisme symplectique
   $\p \in \Aut(\CM\{ q,p \})$
   tel que 
   $$\p(H+R)=H\ \mod\ \left( I^2 \oplus \CM \right). $$
    En particulier, $H+R$ possède un idéal  invariant isomorphe à $I$.
  \end{theoreme}
  En particulier, tout réprésentant de $H+R$ admet une variété lagrangienne complexe invariante isomorphe à une configuration de $2^n$ plan.

%%%%%%%%%%%%%%%%%%%%%%%%%%%%%
\section{Conditions diophantiennes et op\'erateurs born\'es}
Nous utilisons des notations multi-indicielles
$$z^I=z_1^{I_1}z_2^{I_2}\cdots z_d^{I_d} $$
et $| I|=I_1+\cdots+I_d$. Introduisons le {\em produit de Hadamard} de deux séries~:
 $$\sum_{I \in \NM^d} a_I z^I \star \sum_{I \in \NM^d} b_I z^I :=\sum_{I \in \NM^d} a_Ib_I z^I.  $$ 
 Ce produit intervient naturellement dans notre problème. En effet, si 
 $$H=\sum_{i=1}^n \omega_i p_iq_i$$ alors pour tout $F \in \Mt$, on a 
 $$\{ H,F \}=g \star F $$
 avec 
 $$g :=\sum_{I \in \NM^n,J \in \NM^n }^n (\omega,I)q^Ip^J.$$ 
 Or  $t\{H,F\}$ est l'action infinitésimale du flot de $F$ au temps $t$ sur $H$~:
 $$e^{t\{H,F\}}=H+t\{H,F\}+o(t). $$
 Afin de pouvoir inverser l'action infinitésimale et appliquer le théorème de Kolmogorov abstrait, nous devons inverser le produit de Hadamard par $g$.

 \begin{proposition}
 \label{P::Hadamard}
 L'application
 $$f\star:\CM[[z]] \to \CM[[z]],\ g \mapsto f \star g $$
 envoie $\CM\{ z \}$ dans lui-même si et seulement si $f$ est une série analytique.
 \end{proposition}
 \begin{proof}
 Nécessité.\\ On doit avoir
 $$f \star \sum_{I \in \NM^n}  z^I=f \in \CM\{ z \}.$$
 Suffisance.\\
Le critère d'Hadamard dit que le rayon de convergence $R$ d'une série d'une variable complexe $\sum a_n x^n$ est donné par
$$R=\limsup |a_n|^{-1/| n |} .$$
 L'inégalité
 $$\limsup |a_I|^{-1/| I |}|b_I|^{-1/| I |}  \leq \limsup |a_I|^{-1/| I |}  \limsup |b_I|^{-1/| I |} <+\infty $$
 montre que si $f$ et $g$ sont analytiques dans un polydisque de polyrayon $r$ alors $f\star g$ est analytique dans celui de polyrayon 
 $r^2$.
 \end{proof}
  Considérons la famille de polydisques compacts $K=(K_s)$ avec
 $$K_s=\{ z \in \CM^n: | z_i | \leq s,\ i=1,\dots,2n \} . $$  
 La proposition suivante établit un lien entre condition diophantienne et opérateurs bornés~:
  \begin{proposition}
 \label{P::Diophante}
 Pour tout fonction holomorphe 
 $$f=\sum_{I \in \NM^n} a_I z^I,\ z=(z_1,\dots,z_n)$$
 les conditions suivantes sont équivalentes
 \begin{enumerate}[{\rm i)}]
 \item il existe $k$ tel que $a_I=O(| I |^k)$ ;
 \item l'application $f\star : L^{2,\omega}(K) \to L^{2,\omega}(K),\ g \mapsto f \star g $ est $k$-bornée.
 \end{enumerate}
 \end{proposition}
\begin{proof}
Commençons par i) $\implies$ ii).\\
Soit 
$$g=\sum_{I \in \NM^n} b_I z^I. $$
Considérons l'opérateur différentiel
$$D:L^{2,\omega}(K) \to L^{2,\omega}(K), P \mapsto \sum_{i=1}^n z_i\d_{z_i}.  $$
On a 
$$ \sum_{I \in \NM^n}|I|^k | b_I|^2 s^{2|I|}=| D^k g|_s,\ |I|=\sum_{j=1}^n I_j.$$
D'après les inégalités de Cauchy,  tout opérateur différentiel d'ordre $k$ est $k$-borné dans $C^\omega(K)$. Comme les foncteurs $L^{2,\omega}$ et $C^\omega$ sont homotopes, $D^k$ est un opérateur $k+1$ borné $L^{2,\omega}(K)$.

Nous savons que
$$| z^I|_s=\frac{s^{|I|+1}}{\sqrt{2k+2}}. $$
donc il existe une constante $C$ telle que
$$| f \star g|_s =\sum_{I \in \NM^n} | a_Ib_I|^2 \frac{s^{|I|+1}}{\sqrt{2|I|+2}}\leq C\sum_{I \in \NM^n}|I|^k | b_I|^2 \frac{s^{|I|+1}}{\sqrt{2|I|+2}}=C| D^k g|_s. $$
Ce qui montre  i) $\implies$ ii).\\

\noindent $\neg\, {\rm i)} \implies \neg\, {\rm ii)}$.\\
Supposons que pour tout $k$, il existe $C_k$ tel que
$$| a_I| \geq C_k | I |^k .  $$
On a 
$$| f \star g|_s^2 =\sum_{I \in \NM^n} | a_Ib_I|^2 s^{2|I|}\geq C_k^2\sum_{I \in \NM^n}|I|^k | b_I|^2 s^{2|I|}=C_k^2| D^k g|_s. $$
Or $D^k$ ne peut-être $(k-2)$ borné dans $L^{2,\omega}(K)$ sinon il serait $(k-1)$ borné dans $C^\omega(K)$ donc $f\star$ ne peut être $(k-1)$-borné.
Comme ceci est vrai pour tout $k$, cela achève la démonstration de la proposition.
 \end{proof}

 %%%%%%%%%%%%%%%%%%%%%%%%%%%%%%%%%%%%%%
 \section{D\'emonstration du th\'eor\`eme}
 L'application   du théorème de Kolmogorov abstrait se résume maintenant à un problème d'algèbre linéaire. 
 
Appliquons le théorème avec
$$E \subset C^\omega(K),\ M= C^\omega(K)^{(3)},\ F =I^2 \oplus \CM,$$ 
$\alg$ l'espace des dérivations hamiltoniennes d'ordre $2$~:
 $$\alg:=\left\{ \{ -, h(q,p) \}: h \in \Mt^3 \right\}. $$
 et $G$ les automorphismes symplectiques de $E$.
 
 Nous prenons $a=H$. L'action infinitésimale $\rho(\a)$ est alors donnée par
$$\alg \to M/F,\ \{- ,h \} \mapsto   \{ H+\a, h(t,q,p) \}$$
avec $\a \in I^2$.
Notons $V \subset C^\omega(K)$ le sous-espace vectoriel dont les séries de Taylor sont la forme 
$$\sum_{J} \a_{J}  q^J+\sum_{J} \b_{J} q^J $$
On a les identifications~:
 $$M/F \approx M/(I \oplus \CM) \oplus I/I^2 \approx V \oplus \bigoplus_{i=1}^n V p_iq_i . $$

 Décomposons l'opérateur $\{ H+\a,-\}$ dans chacun des facteurs $V$ et $Vp_i q_i$.
Pour $\a \in I^2$, on a 
$$\{ H+\a,h(q,t)\}=f\star h(q,t) +\{ \a,h \}  $$
et
$$ \{ H+\a, p_iq_i h(q,t) \}=p_iq_i f \star h.$$
L'opérateur $\{ H,-\}$ se décompose en bloc sous la forme
$$\begin{pmatrix} f \star & 0 \\
\{ \a ,- \}& f\star
\end{pmatrix}
$$
Posons
$$g=\sum_{I \in \NM^n \setminus \{ 0\} } \frac{1}{(\omega,I)} q^I.$$
L'inverse de $\{H,-\}$ est donné par
$$
\begin{pmatrix} g \star & 0 \\
-g \star \{ \a ,- \} g \star& g\star
\end{pmatrix}
$$
   Chacune des composantes de cette matrice est un opérateur borné, nous avons ainsi trouvé un inverse borné.  D'après le théorème de Kolmogorov, pour tout $R \in \Ot_{\CM^{2n},0}$, il existe un automorphisme symplectique  tel que
  $$\p(H+R)   =H \ \mod \left(\CM \oplus I^2 \right). $$
  Ceci achève la démonstration du théorème.

%%%%%%%
 %%%%%%%%%
  \chapter{Le th\'eor\`eme des tores invariants}
  \section{Tore r\'eel et tore complexe}
  Le théorème des tores d'invariants se démontre de façon analogue au cas singulier, nous devons simplement ajouter un paramètre perturbatif, prendre en compte la structure réelle et remplacer les germes en l'origine par des germes le long d'un tore.

  La variété
 $$ X=(\CM^*)^n \times \CM^n \times \CM.$$
 s'identifie au produit fibré cotangent au tore complexe $(\CM^*)^n$ par une droite complexe. Notons $q_1,\dots,q_n$ les coordonnées
 dans $(\CM^*)^n$, $p_1,\dots,p_n$ celle de $\CM^n$ et $t \in \CM$ le paramètre perturbatif que nous noterons parfois $p_{n+1}$.
 
 La deux forme
 $$\omega:=\frac{1}{\sqrt{-1}}\sum_{i=1}^n \frac{dq_i}{ q_i} \w dp_i $$
 définit une forme symplectique sur les fibres de la projection
 $$X \to \CM,\ (q,p,t) \mapsto t $$
 et, par conséquent, un crochet de Poisson $\CM\{ t\}$-linéaire.
 
 Les seuls crochets de Poisson non nuls sur $X$ entre formes linéaires sont donc
 $$\{ q_i,p_i \}=\sqrt{-1}q_i. $$
 
 La structure réelle sur $X$ est donnée par l'involution antiholomorphe
 $$ q_i \mapsto \frac{1}{\bar q_i},\ p_i \mapsto \bar p_i, t \mapsto \bar t.$$
 La partie réelle de $X$ est
 $$ X(\RM)=(S^1)^n \times \RM^n \times \RM.$$
  On posant $q_i=e^{\sqrt{-1}\p_i}$ pour $q \in (S^1)^n$, on obtient
  $$\omega= \sum_{i=1}^n  d \p_i \w dp_i , $$
 ce qui correspond bien à la forme symplectique du fibré cotangent au tore.
  
   La famille de compacts $K$
 $$K_s=\{ (q,p,t) \in X: 1-s \leq | q_i| \leq 1+s,\ | p_i | \leq s,\ | t | \leq s^2 \}  $$
 définit un voisinage de $X(\RM)$ dans $X$.

 %%%%%%%%%%%%%%%%%%%%%%%%%%%%%%%%%%%%%%
 \section{D\'emonstration du th\'eor\`eme des tores invariants}
 Appliquons  le théorème de Kolmogorov abstrait avec $E \subset C^\omega(K) $ le sous-espace des fonctions réelles, $F \subset E $ la somme du carré de l'idéal engendré par les $p_1,\dots,p_n$ avec
 les fonctions de $t$ et $M \subset E^{(1)}$ les fonctions qui s'annulent en $t=0$~:
 $$M=t\,E.$$
 Nous prenons $\alg \subset \Bt^1(E)^{(2)}$ égal à l'espace des dérivations hamiltoniennes qui s'annulent en $t=0$~:
 $$\alg:=\left\{ \{ -, h(t,q,p) \} +\sum_i a_i(t)\d_{p_i}: a_i(0)=0, h(0,-)=0. \right\}. $$
L'exponentielle d'un élément de  $\alg$ est un symplectomorphisme dépendant de $t$.

Calculons maintenant l'application $\rho$.  La deuxième composante de $\rho$ s'identifie donc à l'application
$$\alg \to M/F,\ \{- ,h \}+\sum_i a_i(t)\d_{p_i} \mapsto   \{ H, h(t,q,p) \} +\sum_i a_i(t)\d_{p_i}H$$
 Notons $M_{q,t}$ (resp. $M_t$) le sous-espace des fonctions de $M$ ne dépendant que de $q,t$ (resp. de $t$). 
 
 Considérons l'application qui a une fonction associe sa valeur moyenne
$$M_{q,t} \to M_t,\ \sum_{I \in \ZM^n} a_i(t)q^i \mapsto a_0  $$

   L'espace vectoriel $M/F$ s'identifie à
 $$ K_{q,t}/ \oplus I/I^2. $$
 et
$$  I/I^2 \approx \bigoplus_{i=1}^n p_i M_{q,t} . $$
Décomposons l'opérateur $\{ H+\a,-\}$ dans chacun des facteurs $M_{q,t}$ et $p_i M_{q,t}$.
Pour $\a \in I^2$, on a 
\begin{align*}
\{ H+\a,h(q,t)\}&=f\star h(q,t) +\{ \a,g \} ;\\
   \{ H+\a, p_i h(q,t) \}&=p_i f \star h(q,t)+\sqrt{-1}q_i (h(q,t)\d_{q_i}H+\d_{q_i} \a)+p_i\{ \a,h(q,t)\}=p_i f \star h(q,t).
\end{align*} 
Par ailleurs~:
$$\d_{q_i} H=0, \d_{q_i} \a \in I^2, \{ \a,h\} \in I$$
 donc on a~:
   $$ \{ H+\a, p_i h(q,t) \} =p_i f \star h(q,t)\ \mod I^2.$$
  
 Donc en utilisant la décomposition
 $$K_{q,t} \oplus \left( \bigoplus_{i=1}^n p_i M_{q,t} \right)  ,  $$
 l'opérateur $\{ -,H \}$ est de la forme
 $$\begin{pmatrix} f \star & 0   \\
\{ \a ,- \}& f\star \\
\end{pmatrix}$$
 
Le produit de Hadamard
$$g=\sum_{I \in \NM^n \setminus \{ 0\} } \frac{1}{(\omega,I)} q^I.$$
est un inverse de $f \star$ sur $K_{q,t}$. Pour trouver un inverse à droite sur la deuxième composante, nous utilisons l'isomorphisme
$$ \bigoplus_{i=1}^n p_i M_{q,t} \approx  \bigoplus_{i=1}^n p_i K_{q,t} \oplus  \bigoplus_{i=1}^n p_i M_{t}.$$
  Sur la première composante, le produit de Hadamard avec $g$ donne l'inverse à droite alors que sur la deuxième, nous utilisons l'hypothèse de non-dégénérescence. Celle-ci implique en effet  que l'application
 $$A:  \sum_{i=1}^na_i(t) \d_{p_i} \mapsto \sum_{i=1}^na_i(t) \overline{\d_{p_i} (H+\a)} $$
 est inversible car
 $$\sum_{i=1}^na_i(t) \overline{\d_{p_i} (H+\a)}=\begin{pmatrix} a_1, \dots ,a_n \end{pmatrix} Hess(H+\a) \begin{pmatrix} p_1\\ \dots \\ p_n \end{pmatrix} $$
 où $Hess$ désigne la hessienne par rapport aux variables $p_1,\dots,p_n$. Nous avons donc trouvé un inverse à droite borné de l'action infinitésimale.
  
  D'après le théorème de Kolmogorov abstrait, pour tout $R \in \Rt_{X,K}$, il existe un automorphisme de Poisson
 tel que
  $$\p(H+R) =H\ \mod  \RM\{ t \} \oplus I^2. $$
  Ceci achève la démonstration du théorème des tores invariants. 
 
 \chapter*{Bibliographie}
 Poincaré découvrit la non-intégrabilité d'abord dans le cas du problème des trois corps dans son article~:\\
 {\sc H. Poincar\'e}, {\em Sur le probl\`eme des trois corps et les \'equations de la dynamique}, {Acta mathematica},	 1890, 13:1, p. 3-270.

Fermi généralisa le théorème de Poincaré dans~:\\
{\sc E. Fermi}, {\em Dimostrazione che in generale un sistema meccanico \`e quasi ergodico,} Il Nuovo
Cimento, 25, p. 267–269, 1923.

voir aussi~:\\
{\sc G. Benettin, G. Ferrari, L. Galgani L et  A. Giorgilli}, {\em An extension of the Poincaré-Fermi theorem on the nonexistence of invariant manifolds in nearly integrable Hamiltonian systems.} Il Nuovo Cimento B, 72:2, p. 137-148, 1982.
 
Le théorème des tores invariants fut annoncé par Kolmogorov dans~:
{\sc A.N. Kolmogorov}, {\em On the conservation of quasi-periodic motions for a small perturbation of the Hamiltonian function}, Dokl. Akad. Nauk SSSR, 98, p. 527-530, 1954.

Contrairement à d'autres articles de cette revue, l'article ne fut pas traduit~(on trouve toutefois une traduction anglaise dans les \oe uvres complètes de Kolmogorov). Ce n'est qu'avec le travail d'Arnold que l'on compris l'importance de cette note~:\\
{\sc V.I. Arnold}, {\em Proof of a theorem of A. N. Kolmogorov on the preservation of conditionally periodic motions under a small perturbation of the Hamiltonian,} Uspehi Mat. Nauk, 18:5, p.13-40, 1963.\\
\rule[0.15cm]{2cm}{0.1pt}{\em\ Small denominators and problems of stability of motion in classical and celestial mechanics},  18:6, p.91-192, 1963.
  
  Les deux articles furent traduits dans les Russian Mathematical Surveys. Cependant la démonstration donnée par Arnold était si confuse, qu'il conseilla au lecteur de construire la sienne en utilisant les idées de son article. On comprend aujourd'hui pourquoi cette  preuve était si complexe~: elle avait un demi-siècle d'avance sur les mathématiques de son époque. \\
 
 Moser proposa d'inclure la théorie dans un schéma général, voir par exemple~:\\
{\sc J. Moser}, {\em A rapidly convergent iteration method and non-linear partial differential equations II}, {Ann. Scuola Norm Sup. Pisa - Classe di Scienze S\'er. 3},{20:3}, p. {499-535}, 1966.
 
Ce travail donna lieu à de nombreux développements (Hamilton, Herman, Zehnder) résumés dans~:\\
{\sc J.-B. Bost}, {\em Tores invariants des systèmes dynamiques hamiltoniens}, Séminaire Bourbaki, 27 (1984-1985), Exposé No. 639.\\

Ce séminaire Bourbaki est caractéristique de la manière de penser la théorie KAM pendant les années quatre-vingts~:
les auteurs se placent dans le cadre de la géométrie différentielle, l'algorithme de Kolmogorov est remplacé par un algorithme de Newton dont on tente de contrôler la convergence par des méthodes subtiles d'analyse réelle.  Le théorème des fonctions implicites ne s'applique pas directement pour deux raisons. Tout d'abord le linéarisé de l'action n'étant inversible que pour les fréquences diophantiennes et non sur un ouvert et, de plus, les hypothèses du théorème rendent impossible l'utilisation de l'exponentielle. Par conséquent, bien que le point de vue géométrique de Moser, que l'on trouve en germe dans Poincaré, soit rendu nettement plus précis après ces travaux, il demeure une heuristique. C'est d'ailleurs toujours le cas si l'on ne fait que des hypothèses $C^\infty$.
  
Pour le point de vue développé dans cet ouvrage~:\\
{\sc J. F\'ejoz et M. Garay}, {\em Un th\'eor\`eme sur les actions de groupes de dimension infinie}, {Comptes Rendus \`a l'Acad\'emie des Sciences}, {348}, {7-8}, p. {427-430}, 2010.\\
{\sc M. Garay}, {\em Degenerations of invariant Lagrangian manifolds}, {Journal of Singularities}, 8, {50-67}, {2014}.\\ 
\rule[0.15cm]{1.8cm}{0.1pt} {\em An Abstract KAM Theorem,} Moscow Math. Journal, 14:4, p.745-772,  2014.
 
 Cette approche est inspirée des méthodes classiques de la géométrie algébrique et analytique, voir par exemple~:\\
 {\sc A. Douady}, {\em Le probl\`eme des modules pour les sous-espaces analytiques compacts d'un espace analytique donn\'e}, {Annales de l'Institut Fourier}, 16, p. 1-95, 1966.
  
 Le théorème affirmant que l'ensemble des vecteurs diophantiens est un ensemble de mesure pleine   a été généralisé dans de nombreuses directions voir par exemple~:\\
 {\sc A.S. Pyartli}, {\em Diophantine approximations on submanifolds of Euclidean space}, Functional Analysis and Its Applications, 3:4, p.303-306,
 1969.\\
 {\sc D.Y. Kleinbock et G.A. Margulis}, {\em Flows on homogeneous spaces and Diophantine approximation on manifolds},
Annals of Mathematics, 148, p.339-360, 1998.\\
{\sc M. Garay},  {\em Arithmetic density}, \`a paraître dans Proceedings of the Edinburgh Mathematical Society, (ArXiv: 1204.2493).

{\bf Crédits photographiques.}
{\em J. Moser} par Konrad Jacobs, Erlangen (1969).\\

%%%%%%%%%%%%%%%%%%%%%%%%%%%%%%%%%
 \end{document}